\documentclass[a4paper, 11pt,reqno]{amsart}

\usepackage[T1]{fontenc}      % Uses 'T1' font encoding.
\usepackage[utf8]{inputenc}
\usepackage[foot]{amsaddr}    % Manages options for displayed author information.
\usepackage{comment}
\usepackage{siunitx}
\usepackage[british]{babel}   % Support for different languages.
\usepackage{csquotes}         % Adds quotation marks according to language.
\usepackage[babel]{microtype} % Makes things look nicer (spacing and so on).

\usepackage{mathtools}        % Extension of 'amsmath' which fixes some bugs.
\usepackage{amssymb}          % Extra symbols and fonts.
\usepackage{amsthm}           % Helps to define "theorem-like" environments.
\usepackage{amstext}          % Defines a '\text' macro for math environments.
\usepackage{mleftright}       % Fixes spacing issue.
\usepackage{gensymb}          % Allows to use degrees.
\usepackage[makeroom]{cancel} % Defines cancellation of terms in formulae.
\usepackage{mathdots}         % Adds some definitions of dots.
\usepackage{mathrsfs}         % Adds rsfs fonts.
\usepackage{dsfont}           % Allows some fancy numbers (\mathds).
\usepackage{stickstootext}
\usepackage[stickstoo,varbb]{newtxmath} % These two lines are an alternative for stix2 which do essentially the same but fix several issues with some symbols...
\makeatletter
\DeclareFontFamily{U}{ntxmia}{\skewchar \font =127}
 \DeclareFontShape{U}{ntxmia}{m}{it}{
                        <-> \ntxmath@scaled ntxmia
                      }{}
                      \DeclareFontShape{U}{ntxmia}{b}{it}{
                        <-> \ntxmath@scaled ntxbmia
                      }{}
\makeatother

\usepackage{graphicx}         % Provides commands to include graphics.
\usepackage{psfrag}           % Gives some extra useful things for graphics.
\usepackage{caption}          % Options for formatting float captions.
\usepackage{tikz}             % Language to create graphics programmatically.
\usetikzlibrary{backgrounds}
\tikzset{vtx/.style={inner sep=1.7pt, outer sep=0pt, circle, fill,draw}}
\usepackage{pgfplots}
\pgfplotsset{compat=1.18}

\usepackage{booktabs}         % Enhances the quality of tables.

\usepackage[margin=1in]{geometry} % Provides commands to define the page layout.
\usepackage{parskip}              % Introduces spacing between paragraphs. Can add options [skip=6pt,indent=12pt]
\newlength\docparskip
\usepackage{enumitem}         % Allows customization for enumerations.
\setlist{nosep, itemsep=0pt, parsep=0pt, before={\parskip=0pt}, after=\vspace{-\docparskip}}%

\usepackage[textsize=footnotesize]{todonotes}
\usepackage[numbers,sort]{natbib}

\makeatletter
\def\NAT@spacechar{~}% NEW
\makeatother

\usepackage{hyperref}              % For hyperlinks within the document
\usepackage[capitalise]{cleveref}  % For improved referencing options; use [capitalise] if desired
\hypersetup{colorlinks=true,
   citecolor=blue,
   filecolor=blue,
   linkcolor=blue,
   urlcolor=blue
}
\usepackage{url}

\makeatletter
\if@cref@capitalise
\crefname{figure}{Figure}{Figures}
\crefname{claim}{Claim}{Claims}
\crefname{conjecture}{Conjecture}{Conjectures}
\crefname{algorithm}{Strategy}{Strategies}
\else
\crefname{figure}{figure}{figures}
\crefname{claim}{claim}{claims}
\crefname{conjecture}{conjecture}{conjectures}
\crefname{algorithm}{strategy}{strategies}
\fi
\makeatother
\Crefname{figure}{Figure}{Figures}
\Crefname{claim}{Claim}{Claims}
\Crefname{conjecture}{Conjecture}{Conjectures}
\Crefname{algorithm}{Strategy}{Strategies}

\usepackage{algorithm}
\usepackage{algpseudocode}
\makeatletter
\renewcommand{\ALG@name}{Strategy} %Change the name Algorithm to Strategy
\makeatother

\allowdisplaybreaks

\newtheorem{theorem}{Theorem}[section]

\newtheorem{corollary}[theorem]{Corollary}

\newtheorem{lemma}[theorem]{Lemma}

\newtheorem{conjecture}[theorem]{Conjecture}

\newtheorem{problem}[theorem]{Problem}

\theoremstyle{definition}

\numberwithin{equation}{section}

\renewcommand{\binom}[2]{\ensuremath{\mleft(\kern-.1em\genfrac{}{}{0pt}{}{#1}{#2}\kern-.1em\mright)}}    % This makes binomial numbers nicer with stix2 (in displayed equations). Remove if stix2 is not loaded.
\newcommand{\inbinom}[2]{\ensuremath{\bigl(\kern-.1em\genfrac{}{}{0pt}{}{#1}{#2}\kern-.1em\bigr)}} % This is better for inline equations, as it will keep sizes of parentheses consistent and not create extra vertical space.

\newcommand*\nume{\ensuremath{\mathrm{e}}}

\newcommand{\cE}{\mathcal{E}}
\newcommand{\cF}{\mathcal{F}}

\newcommand{\cS}{\mathcal{S}}

\DeclareSymbolFont{sfletters}{OML}{cmbrm}{m}{it}
\DeclareMathSymbol{\salpha}{\mathord}{sfletters}{"0B}
\DeclareMathSymbol{\sbeta}{\mathord}{sfletters}{"0C}
\DeclareMathSymbol{\sgamma}{\mathord}{sfletters}{"0D}
\DeclareMathSymbol{\sdelta}{\mathord}{sfletters}{"0E}
\DeclareMathSymbol{\sepsilon}{\mathord}{sfletters}{"0F}
\DeclareMathSymbol{\szeta}{\mathord}{sfletters}{"10}
\DeclareMathSymbol{\seta}{\mathord}{sfletters}{"11}
\DeclareMathSymbol{\stheta}{\mathord}{sfletters}{"12}
\DeclareMathSymbol{\siota}{\mathord}{sfletters}{"13}
\DeclareMathSymbol{\skappa}{\mathord}{sfletters}{"14}
\DeclareMathSymbol{\slambda}{\mathord}{sfletters}{"15}
\DeclareMathSymbol{\smu}{\mathord}{sfletters}{"16}
\DeclareMathSymbol{\snu}{\mathord}{sfletters}{"17}
\DeclareMathSymbol{\sxi}{\mathord}{sfletters}{"18}
\DeclareMathSymbol{\spi}{\mathord}{sfletters}{"19}
\DeclareMathSymbol{\srho}{\mathord}{sfletters}{"1A}
\DeclareMathSymbol{\ssigma}{\mathord}{sfletters}{"1B}
\DeclareMathSymbol{\stau}{\mathord}{sfletters}{"1C}
\DeclareMathSymbol{\supsilon}{\mathord}{sfletters}{"1D}
\DeclareMathSymbol{\sphi}{\mathord}{sfletters}{"1E}
\DeclareMathSymbol{\schi}{\mathord}{sfletters}{"1F}
\DeclareMathSymbol{\spsi}{\mathord}{sfletters}{"20}
\DeclareMathSymbol{\somega}{\mathord}{sfletters}{"21}

\DeclareFontEncoding{LGR}{}{}
\DeclareSymbolFont{sfgreek}{LGR}{cmss}{m}{n}
\SetSymbolFont{sfgreek}{bold}{LGR}{cmss}{bx}{n}
\DeclareMathSymbol{\ssalpha}{\mathord}{sfgreek}{`a}
\DeclareMathSymbol{\ssbeta}{\mathord}{sfgreek}{`b}
\DeclareMathSymbol{\ssgamma}{\mathord}{sfgreek}{`g}
\DeclareMathSymbol{\ssdelta}{\mathord}{sfgreek}{`d}
\DeclareMathSymbol{\ssepsilon}{\mathord}{sfgreek}{`e}
\DeclareMathSymbol{\sszeta}{\mathord}{sfgreek}{`z}
\DeclareMathSymbol{\sseta}{\mathord}{sfgreek}{`h}
\DeclareMathSymbol{\sstheta}{\mathord}{sfgreek}{`j}
\DeclareMathSymbol{\ssiota}{\mathord}{sfgreek}{`i}
\DeclareMathSymbol{\sskappa}{\mathord}{sfgreek}{`k}
\DeclareMathSymbol{\sslambda}{\mathord}{sfgreek}{`l}
\DeclareMathSymbol{\ssmu}{\mathord}{sfgreek}{`m}
\DeclareMathSymbol{\ssnu}{\mathord}{sfgreek}{`n}
\DeclareMathSymbol{\ssxi}{\mathord}{sfgreek}{`x}
\DeclareMathSymbol{\ssomicron}{\mathord}{sfgreek}{`o}
\DeclareMathSymbol{\sspi}{\mathord}{sfgreek}{`p}
\DeclareMathSymbol{\ssrho}{\mathord}{sfgreek}{`r}
\DeclareMathSymbol{\sssigma}{\mathord}{sfgreek}{`s}
\DeclareMathSymbol{\sstau}{\mathord}{sfgreek}{`t}
\DeclareMathSymbol{\ssupsilon}{\mathord}{sfgreek}{`u}
\DeclareMathSymbol{\ssphi}{\mathord}{sfgreek}{`f}
\DeclareMathSymbol{\sschi}{\mathord}{sfgreek}{`q}
\DeclareMathSymbol{\sspsi}{\mathord}{sfgreek}{`y}
\DeclareMathSymbol{\ssomega}{\mathord}{sfgreek}{`w}
\DeclareMathSymbol{\ssvarsigma}{\mathord}{sfgreek}{`c}
\DeclareMathSymbol{\ssGamma}{\mathalpha}{sfgreek}{`G}
\DeclareMathSymbol{\ssDelta}{\mathalpha}{sfgreek}{`D}
\DeclareMathSymbol{\ssTheta}{\mathalpha}{sfgreek}{`J}
\DeclareMathSymbol{\ssLambda}{\mathalpha}{sfgreek}{`L}
\DeclareMathSymbol{\ssXi}{\mathalpha}{sfgreek}{`X}
\DeclareMathSymbol{\ssPi}{\mathalpha}{sfgreek}{`P}
\DeclareMathSymbol{\ssSigma}{\mathalpha}{sfgreek}{`S}
\DeclareMathSymbol{\ssUpsilon}{\mathalpha}{sfgreek}{`U}
\DeclareMathSymbol{\ssPhi}{\mathalpha}{sfgreek}{`F}
\DeclareMathSymbol{\ssPsi}{\mathalpha}{sfgreek}{`Y}
\DeclareMathSymbol{\ssOmega}{\mathalpha}{sfgreek}{`W}

\newcommand{\PP}{\mathbb{P}}
\renewcommand{\Pr}{\PP}

\newcommand{\nc}{\operatorname{nc}}

\DeclareRobustCommand{\msf}[1]{%
  \ifcat\noexpand#1\relax\msfgreek{#1}\else\mathsf{#1}\fi
}
\makeatletter
\newcommand{\msfgreek}[1]{\csname ss\expandafter\@gobble\string#1\endcsname}
\makeatother
\newcommand{\vect}[1]{\boldsymbol{\msf{#1}}}

\makeatletter
\def\moverlay{\mathpalette\mov@rlay}
\def\mov@rlay#1#2{\leavevmode\vtop{%
  \baselineskip\z@skip \lineskiplimit-\maxdimen
  \ialign{\hfil$\m@th#1##$\hfil\cr#2\crcr}}}
\newcommand{\charfusion}[3][\mathord]{
    #1{\ifx#1\mathop\vphantom{#2}\fi
        \mathpalette\mov@rlay{#2\cr#3}
      }
    \ifx#1\mathop\expandafter\displaylimits\fi}
\makeatother
\newcommand{\cupdot}{\charfusion[\mathbin]{\cup}{\cdot}}
\newcommand{\mybar}[1]{\makebox[0pt]{$\phantom{#1}\overline{\phantom{#1}}$}#1}

\renewcommand{\deg}{\operatorname{deg}}

\newcommand{\eps}{\epsilon}

\newcommand{\defi}[1]{\emph{\color{red!60!black}#1}}

\newcommand{\COMMENT}[1]{}
\newcommand{\COMNEW}[1]{}
\renewcommand{\COMNEW}[1]{\footnote{\textcolor{red!70!black}{#1}}} % comment out to hide comments

\title[On constructing small subgraphs in the budget-constrained random graph process]{On constructing small subgraphs\\ in the budget-constrained random graph process}

\author[S.~Antoniuk]{Sylwia Antoniuk}
\email{antoniuk@amu.edu.pl}
\address[Antoniuk]{Faculty of Mathematics and Computer Science, Adam Mickiewicz University, Pozna\'n, Poland.}
\author[A.~Espuny D\'iaz]{Alberto Espuny D\'iaz}
\email{aespuny@ub.edu}
\address[Espuny D\'iaz]{Departament de Matem\`atiques i Inform\`atica, Universitat de Barcelona (UB), Gran Via de les Corts Catalanes, 585, 08007 Barcelona, Spain.}
\author[K.~Petrova]{Kalina Petrova}
\email{kalina.petrova@ist.ac.at}
\address[Petrova]{Institute of Science and Technology Austria (ISTA), Am Campus 1, 3400 Klosterneuburg, Austria.}
\author[M.~Stojakovi\'{c}]{Milo\v{s} Stojakovi\'{c}}
\email{milos.stojakovic@dmi.uns.ac.rs}
\address[Stojakovi\'{c}]{Department of Mathematics and Informatics, Faculty of Sciences, University of Novi Sad, Serbia.}

\thanks{This research was supported by the Oberwolfach Research Institute for Mathematics through its Oberwolfach Research Fellows (OWRF) program.
S.~Antoniuk was supported by Narodowe Centrum Nauki, grant 2024/53/B/ST1/00164.
A.~Espuny Díaz was supported by the Deutsche Forschungsgemeinschaft (DFG, German Research Foundation) through project no.\ 513704762.
K.~Petrova was supported by the European Union’s Horizon 2020 research and innovation programme under the Marie Skłodowska-Curie grant agreement No.~101034413 \includegraphics[width=4mm]{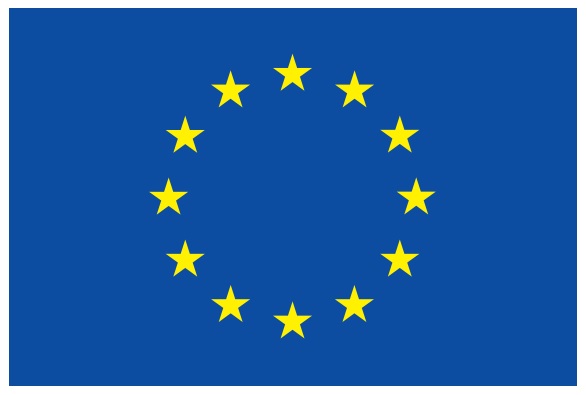}.
M.~Stojakovi\'c was partly supported by the Science Fund of the Republic of Serbia, Grant \#7462: Graphs in Space and Time: Graph Embeddings for Machine Learning in Complex Dynamical Systems (TIGRA), and partly supported by the Ministry of Science, Technological Development and Innovation of the Republic of Serbia (grants 451-03-33/2026-03/200125 \& 451-03-34/2026-03/200125).}

\begin{document}

\begin{abstract}
Consider the budget-constrained random graph process introduced by Frieze, Krivelevich and Michaeli, where each time an edge is offered through the (standard) random graph process we must irrevocably decide whether to ``purchase'' this edge or not, with our goal being to construct a graph which satisfies some property within a given time $t$ and while purchasing at most $b$ edges.
We consider the problem of constructing graphs containing certain fixed small subgraphs.

We provide an optimal strategy for building a graph which contains a copy of $K_4$, showing that budget $b=\omega(\max\{n^8/t^5,n^2/t\})$ suffices and that if $b=o(\max\{n^8/t^5,n^2/t\})$ then no strategy can a.a.s.\ produce a graph containing a copy of $K_4$.
This resolves a problem raised by I{\v{l}}kovi{\v{c}}, Le\'{o}n and Shu.
More generally, we obtain analogously tight results for containing a wheel of any fixed size, or a graph consisting of a tree plus one additional universal vertex.
We also tackle the problem of constructing graphs containing a copy of~$K_5$, obtaining both lower and upper bounds on the optimal budget, though a gap remains in this case.
\end{abstract}

% \vspace{1em}
% Keywords: .

\maketitle

\section{Introduction}

The study of the evolution of different randomised processes for constructing graphs has been a core topic of research in random graph theory.
Arguably the most well-studied process in this family is the so-called \emph{random graph process}.
For a positive integer $n$, let us denote $M\coloneqq\inbinom{n}{2}$ and $[n]\coloneqq\{1,\ldots,n\}$.
The \defi{random graph process} on vertex set $[n]$ refers to a random sequence of graphs $(G_0,G_1,\ldots,G_M)$, all on vertex set $[n]$, where $G_0$ is the empty graph and, for each $i\in[M]$, a pair of vertices $e_i$ is chosen uniformly at random from $\inbinom{[n]}{2}\setminus E(G_{i-1})$ before setting $G_i\coloneqq G_{i-1}\cup\{e_i\}$.
We think of each $i\in[M]$ as a time step, and refer to $e_i$ as the edge \defi{offered}, or \defi{presented}, at time $i$.

We study a variant of this process, proposed recently by \citet{FKM25}, where an intelligent agent called \textbf{Builder} is introduced as a decision agent.
Builder's role is the following: at each time step $i\in[M]$, when the random edge $e_i$ is offered, she must decide (immediately and irrevocably) whether this edge is added to the graph or not.
The motivation for the inclusion of this agent into the model is related to the optimisation of resources throughout the random graph process: indeed, if the aim is to build a graph which satisfies a certain property, it may well be that many of the edges offered throughout the random graph process are not ``useful'' or necessary for attaining the desired property, while adding them to the graph may incur a cost.
As such, Builder's goal will be to find a strategy that allows her to construct a graph which satisfies the desired property, within a certain amount of time, but while purchasing as few edges as possible.

More precisely, we will often assume that Builder has a ``deadline'' for constructing a graph with the desired property, which is expressed through a \defi{time constraint} $t\in[M]$.
Given this constraint, we will denote the sequence of graphs built by Builder by $(B_0,B_1,\ldots,B_t)$; recall that this is defined through an ambient random graph process $(G_0,G_1,\ldots)$, and thus we have the trivial inclusion $B_i\subseteq G_i$ for all~$i\in[t]$.
Moreover, we will assume that there is a limit on the resources that Builder is allowed to use; this is expressed as a \defi{budget constraint} $b\in[t]$, and the strategy that Builder follows must ensure that $|E(B_t)|\leq b$ (note that we may assume that $b\leq t$ as otherwise this constraint would be trivial).
This motivates the name of the \defi{budget-constrained random graph process} when referring to this family of random graph processes.

As is usual when considering random graphs, we are interested in asymptotic statements, that is, we want to know whether Builder has strategies which are very likely to produce a graph with the desired property when $n$ is large.
Formally, we will say that a statement holds \defi{asymptotically almost surely} (a.a.s.\ for short) if the probability that it holds tends to $1$ as $n$ tends to infinity.
Given time and budget constraints $t\in[M]$ and $b\in[t]$, a \defi{$(t,b)$-strategy} is a function which, given a history of the random graph process and the choices of Builder up to some time $0\leq i<t$, and presented with a new edge $e_{i+1}$, outputs whether this next edge should be purchased or not, with the restriction that $|E(B_t)|\leq b$.
We say that a $(t,b)$-strategy $\mathcal{S}$ is \defi{successful} for some (monotone increasing) property~$\mathcal{P}$ if, when running the budget-constrained random graph process under $\mathcal{S}$, a.a.s.\ $B_t\in\mathcal{P}$.
Our first main goal is to determine the optimal asymptotic values of $b$, for each time constraint $t$, for which there exist successful $(t,b)$-strategies for $\mathcal{P}$.
We sometimes informally think of this as a ``budget threshold'' for having successful strategies.

Several papers have considered different spanning properties in the budget-constrained random graph process~\cite{FKM25,Lichev25,Anastos22,KL25,EGNS25a}.
In this paper we instead focus on the case that Builder wants to construct a copy of some fixed subgraph~$F$.
This problem was first addressed by \citet{FKM25}, who obtained tight results on the optimal order of magnitude of~$b$ as a function of~$t$ when attempting to construct copies of any given fixed tree or cycle (and their methods can be used to construct any given unicyclic graph).
Subsequently, \citet{ILS24} obtained analogously tight results for the diamond (the complete graph on four vertices with one edge removed) as well as for $k$-fans (where a $k$-fan for $k\in\mathbb{N}$ is a graph consisting of $k$ triangles, all sharing a single vertex).
No other results for this model were known at the moment.
Both sets of authors asked for the development of general tools to deal with other fixed graphs~$F$, and \citet{ILS24} specifically asked about the case when~$F$ is a clique (with~$K_4$ being the first natural open problem).

A \defi{wheel} on $k\geq4$ vertices, denoted $W_k$, is a graph which consists of a cycle of length $k-1$ with one additional vertex which is joined by an edge to every vertex of the cycle.
Note, in particular, that $W_4=K_4$.
Our first main result in this paper establishes the correct order of magnitude of the ``budget threshold'' required for constructing a copy of $W_k$ in the budget-constrained random graph process, for every $k\geq4$ and the whole range of $t$ (thereby in particular resolving the problem for $K_4$).
This is the first infinite family of graphs containing cycles which share edges for which this threshold is known.
See \cref{fig:wheels} for a visual representation of these results.

\begin{theorem}\label{thm:wheels_main}
    Let $k\geq4$ be an integer.
    For all $t\in[M]$, if
    \begin{equation}\label{equa:Wklowerbound}
        t=o\left(n^{\frac{3}{2}-\frac{1}{2(k-1)}}\right)\qquad\text{ or }\qquad b=o\left(\max\left\{\frac{n^{3k-4}}{t^{2k-3}},\frac{n^{2}}{t}\right\}\right),
    \end{equation}
    then for any $(t,b)$-strategy a.a.s.~$B_t$ does not contain a copy of\/ $W_k$.
    On the other hand, if
    \begin{equation}\label{equa:Wkupperbound}
        t\geq b=\omega\left(\max\left\{\frac{n^{3k-4}}{t^{2k-3}},\frac{n^2}{t}\right\}\right),
    \end{equation}
    then there exists a successful $(t,b)$-strategy for constructing a copy of\/ $W_k$.
\end{theorem}

The lower bound for the optimal budget constraint $b$ expressed in \eqref{equa:Wklowerbound} is a corollary (\cref{cor:Wheel_lower_bound}) of a more general result (\cref{lem:max_no_copies}) which can be applied to obtain some lower bound for any fixed graph~$F$.
However, we do not believe these bounds to be tight in general, and thus do not pursue fully general expressions.
We remark that this addresses, in a weak form, the quest for general tools proposed by \citet{FKM25}; see \cref{sect:lower_bounds} for the general statements, and \cref{sect:conclusions} for a more thorough discussion about this.
The upper bound expressed via \eqref{equa:Wkupperbound} is obtained by exhibiting a specific strategy and analysing its behaviour; the statement pertaining to this bound is reintroduced later as \cref{thm:WheelStrategy}.

\begin{figure}
\begin{tikzpicture}[scale=1]
    \begin{axis}[
    %unit vector ratio*=1 1 1,
    axis lines = middle,
    x=5cm,
    y=5cm,
    xlabel={$\log_{n}t$},
    ylabel={$\log_{n}b$},
    label style = {below left},
    legend style={at={(1.3,0.75)}},
    xmin=4/3, xmax=2.3, 
    xtick={4/3+0.001,11/8,7/5,17/12,3/2,2},
    xticklabels={$4/3$, , , , $3/2$, $2$},
    ymin=0, ymax=1.6,
    ytick={0.001,1/2,4/3,11/8,7/5,17/12},
    yticklabels={$0$, $1/2$, $4/3$,,,}
    ]
        \addplot[blue, ultra thick, domain=4/3:3/2] {8-5*x};
        \addplot[olive, ultra thick, domain=11/8:3/2] {11-7*x};
        \addplot[red, ultra thick, domain=7/5:3/2] {14-9*x};
        \addplot[purple, ultra thick, domain=17/12:3/2] {17-11*x};
        \addplot[purple, ultra thick, domain=3/2:2] {2-x};
        \addplot[gray, thick, domain=4/3:3/2] {x};
        \addplot[gray, thick, dashdotted, domain=4/3:2] {x-1};
        \addplot[gray, dotted, domain=1/2:3/2] {1/2};
        \addplot[gray, dotted, domain=4/3:11/8] {11/8};
        \addplot[gray, dotted, domain=4/3:7/5] {7/5};
        \addplot[gray, dotted, domain=4/3:17/12] {17/12};
        \addplot[gray, dotted] coordinates {(3/2, 0) (3/2, 1/2)};
        \addplot[gray, dotted] coordinates {(11/8, 0) (11/8, 11/8)};
        \addplot[gray, dotted] coordinates {(7/5, 0) (7/5, 7/5)};
        \addplot[gray, dotted] coordinates {(17/12, 0) (17/12, 17/12)};
        \legend{$W_4$,$W_5$,$W_6$,$W_7$,,$x\mapsto x$,$x\mapsto x-1$}
    \end{axis}
\end{tikzpicture}
\caption{A depiction of the optimal budget~$b$ for successful $(t,b)$-strategies for constructing copies of~$W_k$ with $k\in\{4,5,6,7\}$, as given by \cref{thm:wheels_main}.}
\label{fig:wheels}
\end{figure}
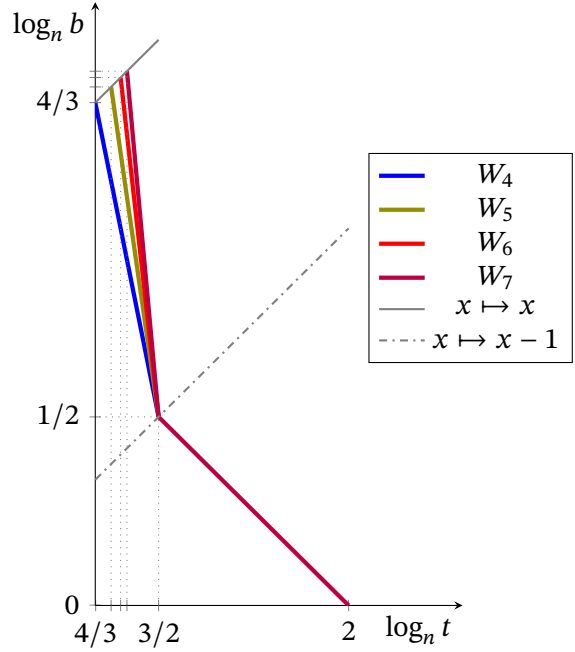

For our second main result, we use our techniques to obtain tight results for a richer family of graphs.
Given a tree $T$, let $K_{1,T}$ denote the graph consisting of a copy of $T$ plus one additional vertex which is joined by an edge to all vertices of $T$.

\begin{theorem}\label{thm:KT_main}
    Let $T$ be a fixed tree.
    Let $m\coloneqq|E(T)|$.
    For all $t\in[M]$, if
    \begin{equation*}\label{equa:KTlowerbound}
        t=o\left(n^{\frac{3m}{2m+1}}\right)\qquad\text{ or }\qquad b=o\left(\max\left\{\frac{n^{3m}}{t^{2m}},\left(\frac{n^{2}}{t}\right)^{\frac{m}{m+1}}\right\}\right),
    \end{equation*}
    then for any $(t,b)$-strategy a.a.s.~$B_t$ does not contain a copy of\/ $K_{1,T}$.
    On the other hand, if
    \begin{equation*}\label{equa:KTupperbound}
        t\geq b=\omega\left(\max\left\{\frac{n^{3m}}{t^{2m}},\left(\frac{n^{2}}{t}\right)^{\frac{m}{m+1}}\right\}\right),
    \end{equation*}
    then there exists a successful $(t,b)$-strategy for constructing a copy of\/ $K_{1,T}$.
\end{theorem}

With this theorem, we recover the result of \citet{ILS24} for the diamond.
It also provides tight results for some relevant families of graphs, such as triangular books.
The proof is analogous to that of \cref{thm:wheels_main}; for completeness, we include it in \cref{appendix}.

Even though our methods lead to tight results for constructing a copy of $K_4$, they do not suffice for obtaining tight results for larger cliques.
Still, they can be used to obtain non-trivial upper and lower bounds on the ``budget threshold''.
As an example, we include our results for the particular case of~$K_5$.
See \cref{fig:K4} for a visual representation of this result, compared with the ``budget thresholds'' for a few other small graphs.

\begin{theorem}\label{thm:K5_main}
    For all $t\in[M]$, if
    \begin{equation}\label{equa:K5lowerbound}
        t=o\left(n^{3/2}\right)\qquad\text{ or }\qquad b=o\left(\max\left\{\frac{n^{15}}{t^{9}},\frac{n^{3}}{t^{3/2}}\right\}\right),
    \end{equation}
    then for any $(t,b)$-strategy a.a.s.~$B_t$ does not contain a copy of~$K_5$.
    On the other hand, if
    \begin{equation}\label{equa:K5upperbound}
        t\geq b=\omega\left(\max\left\{\frac{n^{12}}{t^{7}},\left(\frac{n^{2}}{t}\right)^{5/3}\right\}\right),
    \end{equation}
    then there exists a successful $(t,b)$-strategy for constructing a copy of~$K_5$.
\end{theorem}

For the lower bound~\eqref{equa:K5lowerbound}, we obtain a more general statement which holds for cliques of arbitrary size (see \cref{cor:K4_lower_bound}).
On the other hand, the upper bound~\eqref{equa:K5upperbound} is harder to generalise, and thus here we only consider the case of~$K_5$; see \cref{thm:K5strategy} for the corresponding statement.
We believe that the upper bound should give the correct behaviour.

\begin{conjecture}\label{conj:K5}
    For all $t\in[M]$, if
    \[t=o\left(n^{3/2}\right)\qquad\text{ or }\qquad b=o\left(\max\left\{\frac{n^{12}}{t^{7}},\left(\frac{n^{2}}{t}\right)^{5/3}\right\}\right),\]
    then for any $(t,b)$-strategy a.a.s.~$B_t$ does not contain a copy of~$K_5$.
\end{conjecture}

\begin{figure}
\begin{tikzpicture}[scale=1]
    \begin{axis}[
    %unit vector ratio*=1 1 1,
    axis lines = middle,
    x=5cm,
    y=5cm,
    xlabel={$\log_{n}t$},
    ylabel={$\log_{n}b$},
    label style = {below left},
    legend style={at={(1.3,0.75)}},
    xmin=0, xmax=1.35, 
    xtick={0.001,1/5,1/3,2/5,1/2,3/5,5/8,1},
    xticklabels={$1$, , $4/3$, , $3/2$, , $\ \ \ \ \ \ \ \ 13/8$, $2$},
    ymin=0, ymax=1.7,
    ytick={0.001,1/3,2/5,1/2,3/5,5/8,1,6/5,4/3,3/2},
    yticklabels={$0$, $1/3$, $2/5$, $1/2$, $3/5$,, $1$, $6/5$, $4/3$, $3/2$}
    ]
        \addplot[blue, ultra thick, domain=0:1/3] {1-2*x};
        \addplot[blue, ultra thick, domain=1/3:1] {1/2-x/2};
        \addplot[olive, ultra thick, domain=1/5:2/5] {2-4*x};
        \addplot[olive, ultra thick, domain=2/5:1] {2/3-2*x/3};
        \addplot[red, ultra thick, domain=1/3:1/2] {3-5*x};
        \addplot[red, ultra thick, domain=1/2:1] {1-x};
        \addplot[purple, ultra thick, domain=1/2:5/8] {5-7*x};
        \addplot[purple, ultra thick, domain=5/8:1] {5/3-5*x/3};
        \addplot[purple, ultra thick, dashed, domain=1/2:3/5] {6-9*x};
        \addplot[purple, ultra thick, dashed, domain=3/5:1] {3/2-3*x/2};
        \addplot[gray, thick, domain=0:2/3] {x+1};
        \addplot[gray, thick, dashdotted, domain=0:1] {x};
        \addplot[gray, dotted, domain=0:1/2] {3/2};
        \addplot[gray, dotted, domain=0:1/3] {4/3};
        \addplot[gray, dotted, domain=0:1/5] {6/5};
        \addplot[gray, dotted, domain=0:1/2] {1/2};
        \addplot[gray, dotted, domain=0:1/3] {1/3};
        \addplot[gray, dotted, domain=0:2/5] {2/5};
        \addplot[gray, dotted, domain=0:3/5] {3/5};
        \addplot[gray, dotted, domain=0:5/8] {5/8};
        \addplot[gray, dotted] coordinates {(1/2, 0) (1/2, 3/2)};
        \addplot[gray, dotted] coordinates {(1/3, 0) (1/3, 4/3)};
        \addplot[gray, dotted] coordinates {(2/5, 0) (2/5, 2/5)};
        \addplot[gray, dotted] coordinates {(1/5, 0) (1/5, 6/5)};
        \addplot[gray, dotted] coordinates {(3/5, 0) (3/5, 3/5)};
        \addplot[gray, dotted] coordinates {(5/8, 0) (5/8, 5/8)};
        \legend{$K_3\text{,}\,K_3^+\text{,}\,C_4$,,$K_4^-$,,$K_4$,,$K_5$ upper bound,,$K_5$ lower bound,,$x\mapsto x$,$x\mapsto x-1$}
    \end{axis}
\end{tikzpicture}
\caption{A depiction of the optimal budget $b$ for successful $(t,b)$-strategies for~\textcolor{red}{$K_4$}, as given by \cref{thm:wheels_main}, compared with the optimal budget for its cyclic subgraphs (\textcolor{blue}{$K_3$}, \textcolor{blue}{$K_3^+$} and \textcolor{blue}{$C_4$} are due to \citet{FKM25}, and \textcolor{olive}{$K_4^-$} is due to \citet{ILS24}).
Additional depiction of the upper bound and lower bound on the optimal budget $b$ for successful $(t,b)$-strategies for~\textcolor{purple}{$K_5$}, as given by \cref{thm:K5_main}.}
\label{fig:K4}
\end{figure}
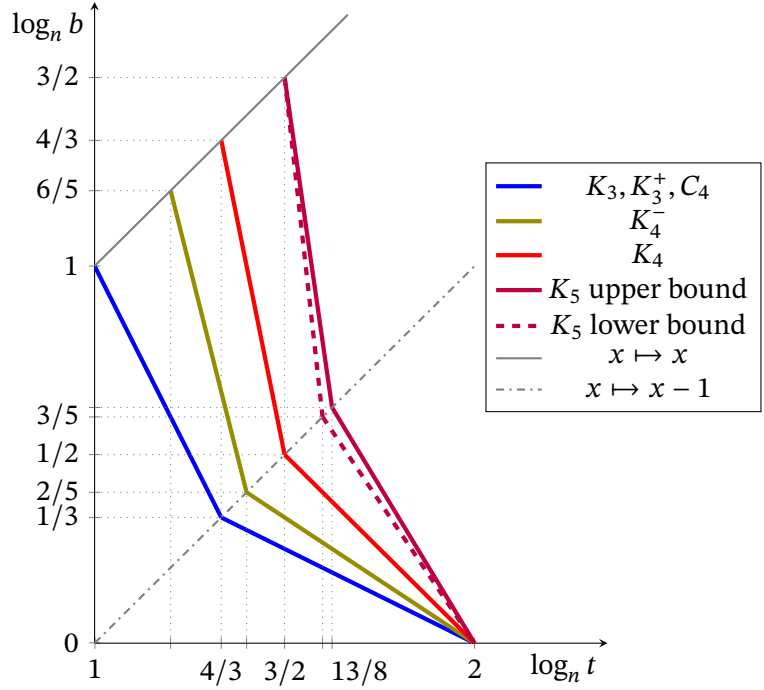

\section{Notation and tools}\label{sect:tools}

Most of our graph theoretic notation is standard.
Given a graph $G$ and a set $U\subseteq V(G)$, we denote by~$G[U]$ the subgraph of $G$ \defi{induced} by $U$, that is, the graph with vertex set $U$ whose edges are all those edges of $G$ which are contained in $U$.
The number of edges of this induced subgraph is denoted by~$e_G(U)$, and we denote $e(G)\coloneqq e_G(V(G))$.
We also denote $v(G)\coloneqq|V(G)|$.
Given a graph~$G$ and a vertex $v\in V(G)$, the \defi{neighbourhood} of $v$ in $G$ is the set $N_G(v)\coloneqq\{u\in V(G):\{u,v\}\in E(G)\}$, and the \defi{degree} of~$v$ is defined as $\deg_G(v)\coloneqq|N_G(v)|$.
The minimum vertex degree of~$G$ is denoted by~$\delta(G)$.
Given two graphs~$G$ and~$H$ on the same vertex set and some set $S\subseteq V(G)$, we denote by $G\setminus H$ the graph on the same vertex set with edge set $E(G)\setminus E(H)$, and by $G-S$ the graph on vertex set $V(G)\setminus S$ obtained from~$G$ by deleting all edges of~$G$ that intersect~$S$.
For a positive integer~$n$ and $p\in[0,1]$, we write $G(n,p)$ to denote the probability distribution on all (labelled) $n$-vertex graphs which results from sampling a graph by including each of its~$M=\inbinom{n}{2}$ possible edges independently with probability~$p$.

We will use the following standard Chernoff bound (see, e.g., the book of \citet[Corollary~2.3 and Theorem~2.10]{JLR}).

\begin{lemma}[Chernoff bound]\label{lem:chernoff}
    Let $X$ be the sum of $n$ independent Bernoulli random variables, or a~hypergeometric random variable, and let $\mu \coloneqq \mathbb{E}[X]$.
    Then, for all $\delta \in [0,1]$, we have that 
    \[\mathbb{P}[|X-\mu|\geq\delta\mu]\leq 2\nume^{-\delta^2\mu/3}.\]
\end{lemma}

Consider the random graph process $(G_0,G_1,\ldots,G_M)$ on vertex set $[n]$.
Let $\tilde{n}=\tilde{n}(n)\in[n]$ and $m\coloneqq\inbinom{\tilde{n}}{2}$, and let $U\subseteq[n]$ be a set of vertices of size $|U|=\tilde{n}$.
While running the random graph process, we can restrict our attention to the offered edges that are contained in $U$ to \defi{simulate} the random graph process on vertex set $U$.
Formally, this is achieved by defining a new sequence of random graphs $(G^U_0,G^U_1,\ldots,G^U_{m})$ as follows.
First, we let $G^U_0$ be the empty graph on vertex set $U$.
Then, while running the random graph process, we let $i_1^U<\ldots<i_{m}^U$ denote the (random) times at which the offered edge $e_{i_j^U}$ is contained in~$U$.
For each $j\in[m]$, we set $G^U_j\coloneqq G^U_{j-1}\cup\{e_{i_j^U}\}$.
It is clear that $(G^U_0,G^U_1,\ldots,G^U_{m})$ is then distributed like the random graph process on $U$.

It is often useful to refer to different \emph{segments} of the random graph process.
The \defi{segment} of the random graph process of length $t\in[M]$ starting at time $j\in\{0\}\cup[M-t]$ is the (random) sequence of graphs $(\tilde{G}_0,\tilde{G}_1,\ldots,\tilde{G}_t)\coloneqq(G_j\setminus G_j,G_{j+1}\setminus G_j,\ldots,G_{j+t}\setminus G_j)$.
We remark that, if the outcome of the random graph process up to time~$j$ has not been revealed, then the segment of the random graph process of length~$t$ starting at time~$j$ has the same distribution as the segment of length~$t$ starting at time~$0$, that is, $(\tilde{G}_0,\tilde{G}_1,\ldots,\tilde{G}_t)\sim(G_0,G_1,\ldots,G_t)$.
Restricting the simulation of the random graph process on a set $U\subseteq[n]$ of size $|U|=\tilde{n}$ to the segment of length $t$ starting at time $j$ leads to a \emph{random} segment of the random graph process on~$U$.
More precisely, the outcome is the sequence of graphs $(\tilde{G}^U_0,\tilde{G}^U_1,\ldots,\tilde{G}^U_{|K|})$ given by taking $K\coloneqq\{k\in[m]:i_k^U\in[j+t]\setminus[j]\}$, letting $k_0\coloneqq\min\{k\in K\}-1$ (or $k_0\coloneqq0$ if $K$ is empty), and defining $\tilde{G}^U_0$ as the empty graph on $U$, and $\tilde{G}^U_\ell\coloneqq\tilde{G}^U_{\ell-1}\cup\{e_{i_{k_0+\ell}^U}\}$ for each $\ell\in[|K|]$.
We will refer to the segment of the random graph process on $U$ obtained when simulating the random graph process on~$U$ during the segment of length~$t$ starting at time~$j$ (that is, the sequence $(\tilde{G}^U_0,\tilde{G}^U_1,\ldots,\tilde{G}^U_{|K|})$ above) as the \defi{$(U,j,t)$-random graph process} (sometimes shortened to $(U,j,t)$-RGP); note, in particular, that the $([n],j,t)$-RGP corresponds to the segment itself.
The \defi{length} of the $(U,j,t)$-RGP, denoted $\ell(U,j,t)$, is the random variable $|K|$.

It will be useful to have some control over the length of the $(U,j,t)$-RGP for different choices of~$U$.
In fact, we will want this in some situations where the set $U$ itself depends on the outcomes of part of the `global' random graph process.\COMMENT{Instead of this, it would suffice if we could prove that the statement holds for every set $U$ of size $\tilde{n}$. However, if we use Chernoff bounds, then we cannot beat the needed union bound. $\checkmark$}
The following lemma encompasses the situations we will encounter later.

\begin{lemma}\label{lem:simulated_length}
    Let $\tilde{n}=\tilde{n}(n)\in[n]$, $t_1=t_1(n)\in[M]$, and $t_2=t_2(n)\in[M]$ be such that $\tilde{n}n\leq t_1=o(n^2)$, $t_1\tilde{n}^2=\omega(n^2)$, and $t_2\tilde{n}^2=\omega(n^2)$.
    Let $j\in[M-t_2]\setminus[t_1]$ and fix a vertex $x\in[n]$.
    Then a.a.s.\ $\deg_{G_{t_1}}(x)\geq \tilde{n}$ and, letting $U$ denote the set of the first $\tilde{n}$ neighbours of $x$ throughout the random graph process, we have that $t_2\tilde{n}^2/2n^2\leq\ell(U,j,t_2)\leq 3t_2\tilde{n}^2/2n^2$.\COMMENT{One can obtain a stronger statement, where one obtains the same conclusion for every set $U$ of size $\tilde{n}$ contained within a neighbourhood of $x$. $\checkmark$}
\end{lemma}

\begin{proof}
    Consider first the $([n],0,t_1)$-RGP; during this segment, we are interested in the neighbourhood of~$x$ and in controlling the number of offered edges contained within this neighbourhood (particularly, within the first $\tilde{n}$ neighbours of $x$ in the random graph process).
    More precisely, let~$X\coloneqq|N_{G_{t_1}}(x)|$, let 
    \[i^*\coloneqq\begin{cases}
        t_1&\text{if }X<\tilde{n},\\
        \min\{i\in[t_1]:|N_{G_i}(x)|=\tilde{n}\}&\text{otherwise},
    \end{cases}\]
    let $N^*\coloneqq N_{G_{i^*}}(x)$ and $X^*\coloneqq|N^*|=\min\{X,\tilde{n}\}$, and let $Y\coloneqq e_{G_{t_1}}(N^*)$.
    Then, define the events $\mathcal{N}\coloneqq\{t_1/n\leq X\leq 3t_1/n\}$, $\mathcal{N}^*\coloneqq\{X^*=\tilde{n}\}\supseteq\mathcal{N}$\COMMENT{Since $t_1\geq n\tilde{n}$, if $\mathcal{N}$ holds, then $X^*=\min\{X,\tilde{n}\}\geq\min\{t_1/n,\tilde{n}\}=\tilde{n}$. $\checkmark$} and $\mathcal{E}\coloneqq\{Y\leq\tilde{n}^2t_1/M\}$.
    
    First, reveal $N_{G_{t_1}}(x)$ (this includes revealing the times at which the edges $xv$ with $v\in N_{G_{t_1}}(x)$ were offered, but not any other edge of the process).
    Note that $\mathbb{E}[X]=(n-1)t_1/M\geq2\tilde{n}$.
    Since $X$ follows a hypergeometric distribution and $\tilde{n}=\omega(1)$\COMMENT{This is ensured by the assumptions, as $t_1=o(n^2)$ but $t_1\tilde{n}^2=\omega(n^2)$. $\checkmark$}, by~\cref{lem:chernoff} we have that $\mathbb{P}[\overline{\mathcal{N}}]\leq\nume^{-\Omega(\tilde{n})}=o(1)$, so~$\mathcal{N}$ holds a.a.s.
    
    Now reveal the remaining edges of the $([n],0,t_1)$-RGP; as the choice of each offered edge is made independently, we have that, upon conditioning on an arbitrary value for $X$, the variable $Y$ follows a~hypergeometric distribution with
    \[\mathbb{E}[Y\mid X]=\binom{X^*}{2}\frac{t_1-X}{M-X}.\]
    In particular, as $t_1/n\geq \tilde{n}=\omega(1)$, for each integer $k\in[t_1/n,3t_1/n]$, conditioning on $X=k$ implies that $X^*=\tilde{n}$ and thus we have that 
    \[\mathbb{E}[Y\mid X=k]=(1\pm o(1))\frac{\tilde{n}^2t_1}{2M}.\]
    As $\tilde{n}^2t_1=\omega(n^2)$, it follows from \cref{lem:chernoff} that 
    $\mathbb{P}[Y>\tilde{n}^2t_1/M\mid X=k]\leq\nume^{-\Theta(\tilde{n}^2t_1/M)}=o(1)$.
    Since~$\mathcal{N}$ holds a.a.s., by the law of total probability, it follows that~$\mathcal{E}$ holds a.a.s.\ too.\COMMENT{This is using the law of total probability: 
    \[\mathbb{P}[\overline{\mathcal{E}}]=\sum_{k=0}^{n-1}\mathbb{P}[\overline{\mathcal{E}}\mid X=k]\mathbb{P}[X=k]\leq o(1)+\sum_{k=\lceil t_1/n\rceil}^{\lfloor 3t_1/n\rfloor}\mathbb{P}[\overline{\mathcal{E}}\mid X=k]\mathbb{P}[X=k]=o(1)+\nume^{-\Theta(\tilde{n}^2t_1/M)}\sum_{k=\lceil t_1/n\rceil}^{\lfloor 3t_1/n\rfloor}\mathbb{P}[X=k]=o(1),\]
    where the first inequality follows since $\mathcal{N}$ holds a.a.s. $\checkmark$}

    Now condition on the event that $\mathcal{N}$ (and thus also $\mathcal{N}^*$) and $\mathcal{E}$ hold, which occurs a.a.s.; from now on, all probabilistic statements refer to this conditional space.
    Note that $U=N^*$ in this conditional space.
    Note, moreover, that the sequence of edges presented during the $([n],j,t_2)$-RGP\COMMENT{Note that we have not verified that $j$ is well defined.
    If $[M-t_2]\setminus[t_1]=\varnothing$ and we cannot choose any $j$, then the statement is vacuously true. $\checkmark$} is a uniformly random sequence of $t_2$ distinct edges from $\inbinom{[n]}{2}\setminus E(G_{t_1})$, which is a set of size $(1-o(1))M$.
    As $\ell(U,j,t_2)$ follows a hypergeometric distribution with
    \[\mathbb{E}[\ell(U,j,t_2)]=\left(\binom{\tilde{n}}{2}-e_{G_{t_1}}(U)\right)\frac{t_2}{M-t_1}=(1\pm o(1))\tilde{n}^2\frac{t_2}{n^2}\] 
    (where the last equality uses the facts that we conditioned on $\mathcal{E}$, that $\tilde{n}=\omega(1)$, and that $t_1=o(n^2)$), \cref{lem:chernoff} implies that
    \[\mathbb{P}\left[\ell(U,j,t_2)\notin\left[\frac{t_2\tilde{n}^2}{2n^2},\frac{3t_2\tilde{n}^2}{2n^2}\right]\right]\leq\nume^{-\Theta(t_2\tilde{n}^2/n^2)}=o(1).\qedhere\]
\end{proof}

In the same way that we can use the `global' random graph process to simulate a random graph process on a~subset $U\subseteq[n]$, even when restricted to segments of the global random graph process, we can also \defi{simulate} any strategy on this subset.
More precisely, we can use the global random graph process to simulate the random graph process on~$U$ and then run a strategy designed for this (simulated) random graph process.
By this we mean that Builder discards every edge offered (by the global random graph process) which is not contained in~$U$, and for edges contained in~$U$ she considers the strategy as if it was run on the (simulated) random graph process on $|U|$ vertices.
Given that, when simulating the random graph process on~$U$ during a segment of the global random graph process, we do not know the length of the resulting segment, we will simulate $(t,b)$-strategies on a~subset $U$ for certain values of~$t$ and~$b$, where we will choose $t$ in such a way that a.a.s.\ the simulated segment has length at least~$t$ (like in  \cref{lem:simulated_length}).
In the unlikely event that the simulated segment has length less than $t$, this strategy is still well defined: it is run until the end of the simulated segment, possibly not reaching its desired outcome, but certainly not going over the (simulated) allotted time or budget.

The following coupling lemma will also be useful in the proofs of our $1$-statements.

\begin{lemma}[{\cite[Lemma~3.3]{EGNS25a}}]\label{lem:coupling}
    Let $k\in\mathbb{N}$ be fixed.
    Let $ t_1, \ldots, t_k \in [M] $ be such that $t_i=\omega(1)$ for every $i\in[k]$ and $t_i=o(M)$ for every $i\in[k-1]$.
    Let $s_0\coloneqq 0$ and, for each $ i \in [k] $, set $ s_i \coloneqq \sum_{j=1}^i t_j $.
    Let $ G_0, G_1, G_2, \ldots $ denote a random graph process on $ [n] $ and, for each $ i \in [k] $, let $ \hat{G}_i \coloneqq G_{s_i} \setminus G_{s_{i - 1}} $.
    
    Then for each $i\in[k]$ there exist $p_i=(1 - o(1)) t_i / M$ and $\mybar{p}_i=o(t_i / M)$ such that there exists a coupling of random graphs $ (H_i, \hat{H}_i, \mybar{H}_i)_{i \in [k]} $ satisfying that
    \begin{enumerate}[label=$(\mathrm{\roman*})$]
        \item\label{lem:couplingproperty1} $ H_i \sim G(n, p_i) $ and $ \mybar{H}_i \sim G(n, \mybar{p}_i) $ for each $i\in[k]$;
        \item\label{lem:couplingproperty2} $(\hat{H}_i)_{i\in[k]}$ is distributed like $(\hat{G}_i)_{i\in[k]}$;
        \item\label{lem:couplingproperty3} the graphs $\{H_i,\mybar{H}_i:i\in[k]\}$ are mutually independent;
        \item\label{lem:couplingproperty4} for each $i\in[k]$, the graphs $H_i$ and $\mybar{H}_i$ are independent of $(\hat{H}_j)_{j\in[i-1]}$, and
        \item\label{lem:couplingproperty5} a.a.s.\ $ H_i\setminus\bigcup_{j=1}^{i-1}\hat{H}_j \subseteq \hat{H}_i \subseteq H_i\cup \mybar{H}_i$ for all $i \in [k] $.
    \end{enumerate}
\end{lemma}

We remark that property \ref{lem:couplingproperty4} above is not explicitly stated in \cite[Lemma~3.3]{EGNS25a}, but it follows immediately from its proof.
Similarly, \cite[Lemma~3.3]{EGNS25a} also imposes an upper bound on $t_k$, but it is not actually needed in its proof.

\section{Lower bounds}\label{sect:lower_bounds}

In this section, we prove our results about lower bounds on ``budget thresholds''.
Throughout the section, when we consider subgraphs of a fixed graph~$F$, or copies of these subgraphs in the random graph process, we always mean \emph{labelled} subgraphs or copies.

The following lemma provides an upper bound on the maximum number of copies of a fixed connected graph~$F$ that Builder can construct (a.a.s.) given any time and budget restrictions~$t$ and~$b$.
The formula that it provides corresponds to the following simple intuition: 
In order to count copies of~$F$, we may root them at some edge and count the number of choices for this edge (for which there are at most $b$ choices).
Then, each other vertex of this copy can be chosen iteratively as a neighbour of a previously fixed vertex (there are roughly at most $\Theta(\min\{b,np\})$ choices for each such vertex, where $p=t/M$, and we must fix $v(F)-2$ other vertices).
This essentially yields an upper bound for the number of copies of a spanning tree of~$F$.
In order for this tree to be completed to a copy of~$F$, each edge of~$F$ that we have not considered so far must at least be offered, which occurs with probability~$p$.
The proof of the lemma formalises this simple intuition.
We remark here that, in order to prove the lower bounds for their results, \citet{ILS24} already used some \emph{ad hoc} arguments for counting copies of subgraphs of the graphs they are trying to construct.
Moreover, a similar intuition to the one discussed above was also key in proving results for $F$-factors in the work of \citet{EGNS25a}.

Let $\nc(F,\cS, n, t, b)$ be the number of copies of a labelled graph~$F$ that Builder purchases when following a~$(t,b)$-strategy $\cS$ in the random graph process on vertex set $[n]$.
Note that $\nc(F,\cS, n, t, b)$ is a~random variable that depends on the outcome of the random graph process.
We write simply $\nc(F)$ when the other parameters are clear from the context.
Moreover, for an edge $e\in E(F)$, let $\nc(F,e,\cS, n, t, b)$, or $\nc(F,e)$ for short, denote the number of copies~$F'$ of~$F$ purchased by Builder for which the edge~$e'\in E(F')$ corresponding to $e$ is purchased last among all the edges in~$E(F')$.

\begin{lemma}\label{lem:max_no_copies}
Let $F$ be a fixed labelled connected graph with $v(F)\geq2$.
Let $t=t(n)\in [M]$ and $b=b(n)\in[t]$.
Set $p\coloneqq t/M$ and let $\eta=\omega(1)$ be a function that grows arbitrarily slowly with $n$.
For every $(t,b)$-strategy~$\cS$, we have that a.a.s.
\[ \nc(F,\cS, n, t, b) \leq \eta \cdot b\cdot \min\{b, np\}^{v(F)-2}p^{e(F)-v(F)+1}.\]
\end{lemma}

\begin{proof}
If $t > M/2$ the bound holds trivially, since $p\in(1/2,1]$ and the maximum number of copies of~$F$ in any subgraph of~$K_n$ with at most~$b$ edges is at most $O(b\cdot\min\{b,n\}^{v(F)-2})$.
From now on, we assume that $t \in [M/2]$.

Let $\gamma\coloneqq \big(\eta / e(F)!\big)^{1/e(F)}=\omega(1)$.
We will show that, for each connected $H \subseteq F$ with $v(H)\geq2$, a.a.s.
\begin{equation}\label{equa:induction_h}
    \nc(H,\cS, n, t, b) \leq e(H)!\cdot  \gamma^{e(H)} \cdot b\cdot \min\{b, np\}^{v(H)-2}p^{e(H)-v(H)+1}
\end{equation}
by induction on the number of edges of $H$.
Note that \eqref{equa:induction_h} with $H\coloneqq F$ gives precisely the conclusion of the lemma.
If~$H$ consists of a single edge, then $\nc(H) \leq b \leq \gamma b$ since Builder can claim at most $b$ edges.
Suppose now that $e(H) > 1$ and that, for all labelled connected graphs $J$ with $1\leq e(J) < e(H)$, we have that a.a.s.\
\begin{equation}\label{equa:induction}
    \nc(J,\cS, n, t, b) \leq e(J)!\cdot  \gamma^{e(J)} \cdot b\cdot \min\{b, np\}^{v(J)-2}p^{e(J)-v(J)+1}.
\end{equation}
For each edge $e\in E(H)$, we are going to bound the number $\nc(H, e)$.
Fix an edge $e\in E(H)$.
We can distinguish three different cases.

\textbf{Case 1.}
Suppose $e$ is a pendant edge, that is, we have $e=uv$ with $d_H(v)=1$.
Let $J\coloneqq H-v$, and note that $J$ is a connected graph.
For each copy of~$J$, there are at most~$b$ ways for Builder to extend it to a copy of~$H$, so 
\begin{equation}\label{eq:pending_edge1}
    \nc(H,e) \leq b\cdot \nc(J).
\end{equation}
On the other hand, at each step $i \leq t$ of the random graph process, conditioned on $(G_0, \dots, G_{i-1})$, the expected value of the number $X_i$ of copies of~$H$ completed by the edge offered at time~$i$ and such that it plays the role of~$e$ in them is
\[ \mathbb{E}[X_i] \leq \nc(J, \cS,n, t,b) \frac{n}{M-i+1} \leq \frac{2n \cdot \nc(J)}{M},\]
where we used that $t \leq M/2$ and that at most $n$ of the $M-i+1$ remaining edges would complete a given copy of $J$ to a copy of $H$.
Since $\nc(H,e) \leq \sum_{i=1}^t X_i$, we conclude
that $ \mathbb{E}[\nc(H,e)] \leq 2 n \cdot t \nc(J) /M = 2 n \cdot \nc(J) \cdot p$.
By Markov's inequality, it follows that 
\begin{equation}\label{eq:pending_edge2}
\Pr[\nc(H,e) \geq \gamma \nc(J)\cdot np] \leq 2/\gamma=o(1).
\end{equation}
Combining \eqref{eq:pending_edge1} and \eqref{eq:pending_edge2} and using \eqref{equa:induction}, we thus have that a.a.s.
\begin{align}\label{eq:pending_edge}
    \nc(H, e) & \leq \gamma \cdot \nc(J) \cdot \min\{b, np\} \nonumber \\
    & \leq \gamma \cdot e(J)! \cdot \gamma^{e(J)}\cdot b\cdot \min\{b, np\}^{v(J)-2}p^{e(J)-v(J)+1}  \cdot \min\{b,np\} \nonumber \\
    & =\big(e(H)-1\big)!\cdot  \gamma^{e(H)}\cdot b\cdot \min\{b, np\}^{v(H)-2}p^{e(H)-v(H)+1}.
\end{align}

\textbf{Case 2.} Suppose $e$ is a cut-edge, but not a pendant edge.
Then $H\setminus\{e\}$ consists of two connected components, each containing at least one edge; we denote these components as~$J_1$ and~$J_2$.
As $b\leq t\leq n^2p$, we have that $bp \leq \min\{b^2, (np)^2\}$.
By a similar argument as in the previous case, since each copy of~$H$ in~$G_t$ completed by an edge corresponding to~$e$ must be obtained from one copy of~$J_1$ and one copy of~$J_2$, we have that
\[\mathbb{E}[\nc(H,e)]\leq 2t\frac{\nc(J_1)\cdot\nc(J_2)}{M}=2\nc(J_1)\cdot \nc(J_2)\cdot p.\]
Therefore, by Markov's inequality and \eqref{equa:induction}, a.a.s.
\begin{align}\label{eq:cut_edge}
    \nc(H, e) & \leq \gamma \cdot \nc(J_1) \cdot \nc(J_2) \cdot p \nonumber \\
    &\leq \gamma\cdot e(J_1)! \cdot e(J_2)! \cdot \gamma^{e(J_1) + e(J_2)}\cdot b^2\cdot \min\{b, np\}^{v(J_1)+v(J_2)-4}p^{e(J_1)+e(J_2)-v(J_1)-v(J_2)+3} \nonumber \\
    & \leq \big( e(H)-1\big)! \cdot \gamma^{e(H)}\cdot b^2\cdot \min\{b, np\}^{v(H)-4}p^{e(H)-v(H)+2} \nonumber \\
    & \leq \big( e(H)-1\big)! \cdot \gamma^{e(H)} \cdot b\cdot \min\{b, np\}^{v(H)-2}p^{e(H)-v(H)+1}.
\end{align}

\textbf{Case 3.} Suppose $e$ is not a cut-edge and let $J \coloneqq H\setminus\{e\}$.
In this case, as each copy of $H$ completed by an edge corresponding to $e$ may only arise from some copy of $J$, by a similar argument as in the previous two cases we have that 
\[\mathbb{E}[\nc(H,e)]\leq 2t\frac{\nc(J)}{M}=2\nc(J)\cdot p.\]
Thus, by Markov's inequality and \eqref{equa:induction}, a.a.s.
\begin{align}\label{eq:no_cut_edge}
    \nc(H, e) & \leq \gamma \cdot \nc(J) \cdot p \leq \gamma \cdot e(J)! \cdot \gamma^{e(J)} \cdot b\cdot \min\{b, np\}^{v(J)-2}p^{e(J)-v(J)+2} \nonumber \\
    & = \big( e(H) -1\big)! \cdot \gamma^{e(H)} \cdot b\cdot \min\{b, np\}^{v(H)-2}p^{e(H)-v(H)+1}.
\end{align}

Now, the desired upper bound \eqref{equa:induction_h} on $\nc(H,\cS, n, t, b)$ follows from \eqref{eq:pending_edge}, \eqref{eq:cut_edge} and \eqref{eq:no_cut_edge} by summing over all possible options for $e\in E(H)$.
\end{proof}

\Cref{lem:max_no_copies} can be used to obtain a lower bound on the minimum budget required for constructing copies of any fixed (not necessarily connected) graph~$F$, by simply verifying under which conditions for~$b$ we can guarantee that a.a.s.\ $\nc(F,\cS, n, t, b)=o(1)$ for every $(t,b)$-strategy~$\mathcal{S}$, which implies that a.a.s.\ there are no copies of~$F$.
We believe this lower bound is far from optimal in general, but it \emph{is} optimal when considering wheels.

\begin{corollary}\label{cor:Wheel_lower_bound}
    Let $k\geq4$ be an integer.
    For all $t\in[M]$, if
    \[t=o\left(n^{\frac{3}{2}-\frac{1}{2(k-1)}}\right)\qquad\text{ or }\qquad b=o\left(\max\left\{\frac{n^{3k-4}}{t^{2k-3}},\frac{n^{2}}{t}\right\}\right),\]
    then for any $(t,b)$-strategy a.a.s.~$B_t$ does not contain a copy of\/ $W_k$.
\end{corollary}

\begin{proof}
Let $p\coloneqq t/M$.
If $t=o(n^{3/2-1/2(k-1)})$, then a standard application of the first moment method shows that a.a.s.~$G_t$ does not contain a copy of~$W_k$, and so neither does $B_t$. %classical result of \citet{ER60} ensures that a.a.s.~$G_t$ does not contain a copy of~$W_k$, and so neither does $B_t$.
Thus, assume that \mbox{$t=\Omega(n^{3/2-1/2(k-1)})$} and let~$\mathcal{S}$ be a $(t, b)$-strategy. By \cref{lem:max_no_copies}, for any arbitrarily slowly growing function $\eta=\omega(1)$ we have that a.a.s.
\[ \nc(W_k) =  \nc(W_k,\cS,n,t,b) \leq \eta \cdot b \cdot \min\{b,np\}^{k-2} p^{k-1}.\]

Note that
\[\frac{n^{3k-4}}{t^{2k-3}}=\frac{n^{2}}{t}\iff t=n^{3/2}.\]
When $t=O(n^{3/2})$ we have that $n^{3k-4}/t^{2k-3}=\Omega(n^2/t)$.
For any $b=o(n^{3k-4}/t^{2k-3})$, choose some $\eta=\omega(1)$ with $\eta=o(1/bn^{k-2}p^{2k-3})$ (which exists as $n^{3k-4}/t^{2k-3}=\Theta(1/n^{k-2}p^{2k-3})$).
We then conclude that a.a.s.\COMMENT{We are using the trivial bound $\min\{b,np\}\leq np$. $\checkmark$}
\[ \nc(W_k) \leq \eta \cdot bn^{k-2}p^{2k-3} = o(1) < 1.\]
If instead $t=\omega(n^{3/2})$, then $n^2/t=\omega(n^{3k-4}/t^{2k-3})$.
Similarly as above, for any $b=o(n^2/t)=o(1/p)$, we choose some $\eta=\omega(1)$ with $\eta=o(1/b^{k-1}p^{k-1})$.
We conclude that a.a.s.\COMMENT{We are now using the trivial bound $\min\{b,np\}\leq b$. $\checkmark$}
\[\nc(W_k) \leq \eta \cdot b^{k-1}p^{k-1} = o(1) < 1.\qedhere\]
\end{proof}

While we do not believe that we can obtain results which are tight in general with our approach, it is instructive to see other applications.
Here, for simplicity, we write the result for cliques.
We remark that the case $r=3$ recovers the tight bound for constructing a triangle given by \citet{FKM25}, and the case $r=4$ gives a tight lower bound for the $0$-statement in the case $k=4$ of \cref{thm:wheels_main}.
The special cases $r\in\{3,4,5\}$ of this result are depicted in \cref{fig:K4}.

\begin{corollary}\label{cor:K4_lower_bound}
    Let $r\geq3$ be an integer.
    For all $t\in[M]$, if
    \[t=o\left(n^{2-\frac{2}{r-1}}\right)\qquad\text{ or }\qquad b=o\left(\max\left\{\frac{n^{r(r-2)}}{t^{\binom{r}{2}-1}},\left(\frac{n^{2}}{t}\right)^{(r-2)/2}\right\}\right),\]
    then for any $(t,b)$-strategy a.a.s.~$B_t$ does not contain a copy of $K_r$.
\end{corollary}

\begin{proof}
Let $p\coloneqq t/M$.
If $t=o(n^{2-2/(r-1)})$, then a standard application of the first moment method shows that a.a.s.\ $G_t$ does not contain a copy of~$K_r$, and thus neither does $B_t$. %classical result of \citet{ER60} implies that a.a.s.\ $G_t$ does not contain a copy of $K_r$, and thus neither does $B_t$.
Therefore, assume that \mbox{$t=\Omega(n^{2-2/(r-1)})$} and let~$\mathcal{S}$ be a $(t, b)$-strategy.
By \cref{lem:max_no_copies}, for any arbitrarily slowly growing function $\eta=\omega(1)$ we have that a.a.s.
\[ \nc(K_r) =  \nc(K_r,\cS,n,t,b) \leq \eta \cdot b \cdot \min\{b,np\}^{r-2} p^{\binom{r}{2}-r+1}.\]

Observe that
\[ \frac{n^{r(r-2)}}{t^{\binom{r}{2}-1}} = \left(\frac{n^{2}}{t}\right)^{(r-2)/2}\iff t = n^{2-2/r}.\]
When $t=O(n^{2-2/r})$ we have that $n^{r(r-2)}/t^{\binom{r}{2}-1}=\Omega((n^2/t)^{(r-2)/2})$.
For any $b=o(n^{r(r-2)}/t^{\binom{r}{2}-1})$, choose some $\eta=\omega(1)$ with $\eta=o(1/bn^{r-2}p^{\binom{r}{2}-1})$ (which exists as $n^{r(r-2)}/t^{\binom{r}{2}-1}=\Theta(1/n^{r-2}p^{\binom{r}{2}-1})$).\COMMENT{So when we further divide by $b$, which is asymptotically smaller than this, we obtain something that tends to infinity. $\checkmark$}
We then conclude that a.a.s.\COMMENT{We are using the trivial bound $\min\{b,np\}\leq np$. $\checkmark$}
\[ \nc(K_r) \leq \eta \cdot bn^{r-2}p^{\binom{r}{2}-1} = o(1) < 1.\]
If instead $t=\omega(n^{2-2/r})$, then $(n^2/t)^{(r-2)/2}=\omega(n^{r(r-2)}/t^{\binom{r}{2}-1})$.
Proceeding similarly as above, for any $b=o((n^2/t)^{(r-2)/2})=o(p^{-(r-2)/2})$, take $\eta=\omega(1)$ with $\eta=o(1/b^{r-1}p^{\binom{r}{2}-r+1})$.
We conclude that a.a.s.
\[\nc(K_r) \leq \eta \cdot b^{r-1}p^{\binom{r}{2}-r+1} = o(1) < 1.\qedhere\]
\end{proof}

We end this section with a lemma that is not strictly needed for our results, but could be of independent interest as a tool for proving lower bounds in future work.
Given a fixed graph $F$ and some integer $k \in \{0,1,\dots,e(F)\}$, it allows us to bound the probability that Builder can construct a copy of~$F$ in terms of the number of copies of all subgraphs of~$F$ on~$k$ edges that she can build.
The lemma can be seen as a formalisation and generalisation of the ideas in the lower bound proofs of \citet{ILS24}, in that one counts the number of some `intermediate' subgraphs $H$, which form `traps', and then estimates the probability that any of these traps is completed to a copy of the target graph $F$.

For any (not necessarily connected) labelled graph~$H$, any $\tilde{p} \in (0,1]$, and any non-negative integers $n$, $t$, and $b$, let $f(H,n,t,b, \tilde{p})$ denote the minimum integer~$\tilde{f}$ such that, for any $(t,b)$-strategy~$\cS$, the probability that the budget-constrained random graph process on~$[n]$ under~$\cS$ contains more than~$\tilde{f}$ copies of~$H$ is at most~$\tilde{p}$.

\begin{lemma}\label{lem:lower_bound_probability}
Let $F$ be a fixed labelled graph.
For any $t=t(n)\in[M]$ and $b=b(n)\in[t]$, let $p\coloneqq t/M$, let $\mathcal{S}$ be a $(t, b)$-strategy, and let $\mathcal{E}_F$ be the event that, when considering the budget-constrained random graph process under $\mathcal{S}$, the resulting graph $B_t$ contains a copy of~$F$.
Then, for any integer $k \in \{0,1,\dots,e(F)\}$ and any \mbox{$\tilde{p}=\tilde{p}(n) \in (0,1]$}, we have that
\begin{equation}\label{equa:prob_lower_bound}
    \mathbb{P}[\mathcal{E}_F] = O\Bigg(  \sum_{H\subseteq F,\, e(H)=k,\, \delta(H)\geq 1} f(H,n,t,b, \tilde{p}) n^{v(F)-v(H)}p^{e(F)-k}\Bigg) + O(\tilde{p}).
\end{equation}
\end{lemma}

This lemma can be useful whenever there is some $k$ for which we have some good `with high probability' bounds on the number of copies of $k$-edge subgraphs of $F$ that Builder can build.
There is a clear trade-off between $f(H, n,t,b,\tilde{p})$ and $\tilde{p}$, the probability of the `bad' event that Builder can build many copies of some $k$-edge subgraph $H$.
One possible application of \cref{lem:lower_bound_probability} would be to give an alternative proof of the $r =4$ case of \cref{cor:K4_lower_bound}, by making use of \cref{lem:max_no_copies} to estimate $f(H,n,t,b,\tilde{p})$ for some $\tilde{p} = o(1)$ and for all $H\subseteq K_4$ with $e(H)=4$.
%However, this `alternative' proof is really essentially the same proof as the one given by applying Lemma~\ref{lem:max_no_copies} directly, since both proofs call Lemma~\ref{lem:max_no_copies} on subgraphs $H$ of $K_4$ and use its output in a very similar way, by considering buying~$H$ first, and then going over all options for the step~$i$ at which the next edge (after~$H$) could be bought.
It is conceivable that, for some graphs~$F$, one can obtain better bounds on $f(H,n,t,b,\tilde{p})$ for subgraphs $H$ of $F$ than what \cref{lem:max_no_copies} gives, in which case \cref{lem:lower_bound_probability} can give better results than \cref{lem:max_no_copies}.
We have not been able to find such graphs $F$, but we prove Lemma~\ref{lem:lower_bound_probability} here in the hopes that it could be useful in future work.

\begin{proof}[Proof of \cref{lem:lower_bound_probability}]
For each $H\subseteq F$, let $\cF_H$ be the event that there are at most $f(H,n,t,b,\tilde{p})$ copies of~$H$ in~$B_t$.
Then we have that
\begin{align*}
 \mathbb{P}[\cE_F] &\leq \mathbb{P}\left[\cE_F \cap  \bigcap_{H \subseteq F,\, e(H) = k,\, \delta(H)\geq 1} \cF_H \right] + \sum_{H\subseteq F,\, e(H) = k,\, \delta(H)\geq 1} \mathbb{P}\left[\overline{\cF}_H\right] \\
 &\leq \mathbb{P}\left[\cE_F \cap  \bigcap_{H \subseteq F,\, e(H) = k,\, \delta(H)\geq 1} \cF_H \right]  + \binom{e(F)}{k} \tilde{p}.
\end{align*}
Since $e(F)$ is a constant, we have that $\inbinom{e(F)}{k} \tilde{p} = O(\tilde{p})$, which accounts for the second term of the upper bound in the statement of the lemma.
For the rest of the proof, we focus on bounding the first term in the expression above, which amounts to bounding the probability of $\cE_F$ while being able to assume `for free' that $\cF_H$ holds for all subgraphs $H$ of $F$ with $k$ edges and no isolated vertices.
More specifically, for each subgraph $H$ of $F$ with $k$ edges and no isolated vertices and for each step~$i$ of the random graph process, we will consider revealing the first $i-1$ steps of the process and `rejecting' the outcome if the graph $G_{i-1}$ contains more than $f(H,n,t,b,\tilde{p})$ copies of~$H$.
Thus, in all outcomes that actually count towards our upper bound, we get to assume there are at most $f(H,n,t,b,\tilde{p})$ copies of~$H$ in~$G_{i-1}$.

If $t > M/2$, the conclusion of the lemma holds trivially, since the first term of the right-hand side in~\eqref{equa:prob_lower_bound} is either~$\Omega(1)$ or~$0$.\COMMENT{Under this assumption, if $f(H,n,t,b,
\tilde{p})\geq1$ for at least one choice of $H$, we have that $p$ is constant and so the term inside the $O$ is of order $\Omega(n^{v(F)-v(H)})$. $\checkmark$}
Furthermore, it can only be~$0$ if $f(H,n,t,b,
\tilde{p})=0$ for all subgraphs $H \subseteq F$ with $e(H)=k$ and $\delta(H)\geq1$ and, if that is the case, then indeed no copies of $F$ can be built as long as $\bigcap_{H \subseteq F,\, e(H) =k,\, \delta(H) \geq 1}\cF_H$ holds.
Hence, from now on we assume that $t \leq M/2$.
If $k=e(F)$, the result is also trivial for the same reason, so assume that $k<e(F)$.

Any copy of $F$ that Builder may eventually produce is constructed edge by edge.
We argue by a~union bound over all possible choices for the first $k$ edges of $F$ that Builder can purchase, and all choices for the $(k+1)$-th edge.
Let $H\subseteq F$ be a subgraph with exactly $k$ edges (which will play the role of the first~$k$ edges of~$F$ that Builder purchases) and $e\in E(F)\setminus E(H)$ (which will play the role of the the $(k+1)$-th edge of $F$ purchased by Builder).
Let $H'$ be the spanning subgraph of $F$ with $E(H') = E(H)$ (note that~$H'$ has $v(F) - v(H)$ isolated vertices).
Then note that $f(H',n,t,b,\tilde{p}) \leq f(H,n,t,b,\tilde{p}) n^{v(F) - v(H)}$.
For each step $i\in[t]$ of the random graph process, the number of candidates for copies of $H' \cup \{e\}$ in~$G_i$ such that $e$ is embedded into $G_i \setminus G_{i-1}$ is at most $f(H',n,t,b,\tilde{p})$.
Conditioned on~$G_{i-1}$, each such candidate is present in~$G_i$ with probability $1/(M-i+1)$.
The probability that the remaining $e(F) - k - 1$ edges are presented after time~$i$ and before time $t+1$ is at most\COMMENT{Note this is indeed ``at most'' as some of these edges may already have been presented, in which case this probability is $0$. $\checkmark$}
\[ \frac{\binom{M- i - (e(F)-k-1)}{t- i - (e(F)-k-1)}}{\binom{M-i}{t-i}} \leq \Bigg( \frac{t-i}{M-i}\Bigg)^{e(F)-k-1} \leq p^{e(F) -k - 1}.\]
All in all, by the union bound, we conclude that
\begin{align*}
\Pr\left[\mathcal{E}_F \cap \bigcap_{\substack{H\subseteq F\\ e(H)=k,\, \delta(H) \geq 1}}\cF_H\right] &\leq \sum_{\substack{H\subseteq F\\ e(H)=k,\, \delta(H) \geq 1}} \sum_{e \in E(F\setminus H)} \sum_{i=1}^t f(H,n,t,b,\tilde{p}) n^{v(F)-v(H)} \frac{1}{M-i+1} p^{e(F) - k-1}\\ 
& \leq \sum_{\substack{H\subseteq F\\ e(H)=k,\, \delta(H) \geq 1}} e(F) \cdot t \cdot f(H,n,t,b,\tilde{p}) n^{v(F) - v(H)} \frac{2}{M} p^{e(F)-k-1} \\ 
& \leq \sum_{\substack{H\subseteq F\\ e(H)=k,\, \delta(H) \geq 1}} 2e(F) f(H,n,t,b,\tilde{p}) n^{v(F)-v(H)} p^{e(F)-k},
\end{align*}
as desired.
\end{proof}

%%%%%%%%%%%%%%%%%%%%%%%%%%%%%%%%%%%%%%%%%%%
%%%%%%%%%%%%%%%%%%%%%%%%%%%%%%%%%%%%%%%%%%%

\section{Upper bounds}\label{sect:upper_bounds}

In this section, we provide the proofs for the $1$-statements of our theorems.
In other words, we want to show that, if the budget is sufficiently large (as a function of $t$ and $n$), then there exist successful $(t,b)$-strategies for containing different subgraphs.
We begin by considering wheels.

% \begin{theorem}\label{thm:K4strategy}
%     For all $t=\omega(n^{4/3})$ with $t\leq M$, if 
%     \[b=\omega\left(\max\left\{\frac{n^8}{t^5},\frac{n^2}{t}\right\}\right),\] 
%     there exists a successful $(t,b)$-strategy for constructing a copy of $K_4$.
% \end{theorem}

\begin{theorem}\label{thm:WheelStrategy}
    Let $k\geq4$ be a fixed integer.
    If 
    \[M\geq t\geq b=\omega\left(\max\left\{\frac{n^{3k-4}}{t^{2k-3}},\frac{n^2}{t}\right\}\right),\] 
    there exists a~successful $(t,b)$-strategy for constructing a copy of\/ $W_k$.
\end{theorem}

Our proof of \cref{thm:WheelStrategy} is split into two cases, depending on whether $t=O(n^{3/2})$ or $t=\omega(n^{3/2})$.
The strategies for constructing wheels $W_k$ for $k\geq4$ are rather simple.
It is useful to consider the intuition that the random graph process can be split into a number of segments, each of them behaving roughly like a binomial random graph of the appropriate density, as formalised in \cref{lem:coupling}.

When $t=O(n^{3/2})$, the idea is to first construct a set of stars, each as large as possible, of an appropriate size.
The centres of these stars would play the role of the single vertex of higher degree in $W_k$.
Then, in the second round of exposure, one adds all edges contained in any of the sets of leaves of a star.
If this second round of exposure creates a cycle of length $k-1$ within the set of leaves of some star, this cycle together with the star it is contained in forms a copy of $W_k$.
We formalise all the details for this strategy below.

When $t=\omega(n^{3/2})$, the strategy has more levels of ``depth'' depending on the value of $k$.
Let us illustrate this for the particular case $k=4$.
We fix a single vertex $x$ and, in the first stage, construct a star with $x$ at its centre.
Let us denote the set of leaves of this star by $S$.
In the second round of exposure, we fix a subset $Y\subseteq S$ of an appropriate size, and purchase all edges contained in $S$ which are incident to some vertex in $Y$.
This leads to sets $S_1,\ldots,S_{|Y|}$, denoting the neighbours in~$S$ of vertices of~$Y$, respectively, at the end of this second round.
In the final round of exposure, it suffices to purchase a single edge contained in one of the sets $S_i$ to complete a $K_4$.
To simplify our strategy, we can summarise it as follows: first, we grow a star centred at $x$, and then we apply an optimal strategy for constructing a triangle within the neighbourhood of $x$.
For general $k\geq4$, the summary is the same: we first grow a star centred at $x$ and then apply an optimal strategy for constructing a cycle of length $k-1$ within the neighbourhood of $x$.
As such strategies for cycles have already been studied, our analysis will be greatly simplified by using the following result.

% \begin{lemma}[{\citet[Theorem~1.6]{FKM25}}]\label{lem:TriangleStrat}
%     For all $t=\omega(n)$ with $t\leq M$, if 
%     \[b=\omega\left(\max\left\{\frac{n^3}{t^2},\frac{n}{t^{1/2}}\right\}\right),\] 
%     there exists a successful $(t,b)$-strategy for constructing a copy of $K_3$.
% \end{lemma}

\begin{lemma}[{\citet[Theorem~1.6]{FKM25}}]\label{lem:CycleStrat}
    Let $k\geq4$ be a fixed integer.
    For all $t=\omega(n)$ with $t\leq M$, if 
    \[b=\omega\left(\max\left\{\frac{n^{\lfloor k/2\rfloor+1}}{t^{\lfloor k/2\rfloor}},\frac{n}{t^{1/2}}\right\}\right),\] there exists a successful $(t,b)$-strategy for constructing a copy of $C_{k-1}$.
\end{lemma}

\begin{proof}[Proof of \cref{thm:WheelStrategy}]
    Fix an arbitrary $k\geq4$.
    Following the statement, we may assume throughout that $M\geq t=\omega(n^{(3k-4)/(2k-2)})=\omega(n^{3/2-1/(2k-2)})$.\COMMENT{The lower bound on $t$ follows by reordering the bounds in the statement. $\checkmark$}
    We split our proof into two cases, depending on the range of $t$ that we consider.

    \textbf{Case 1.}
    Assume first that $t=O(n^{3/2})$.
    Note that in this range we have $b=\omega(n^{3k-4}/t^{2k-3})$.\COMMENT{Indeed, we have that $\frac{n^{3k-4}/t^{2k-3}}{n^2/t}=n^{3(k-2)}/t^{2(k-2)}=\Omega(1)$ by the upper bound on $t$. $\checkmark$}
    Fix any such~$b$ and let $r=r(n)$ be such that $r=o(t)$ but it is sufficiently close to $t$ that
    \begin{equation}\label{equa:artbound1}
        \frac{n^{3k-4}t}{r^{2k-2}}=o(b).
    \end{equation}
    Then, consider the strategy outlined in \cref{strat:K4sparse} below.
    This is a $(t,b)$-strategy by construction, so it only remains to prove that it is successful for constructing a copy of~$W_k$.

    \begin{algorithm}
    \caption{A $(t,b)$-strategy for $W_k$ for $t=O(n^{3/2})$.}\label{strat:K4sparse}
    \begin{algorithmic}[1]\setcounter{ALG@line}{-1}
        \State{Set $X\coloneqq[\lceil n^{3k-3}/r^{2k-2}\rceil]$ and $V\coloneqq[n]\setminus X$.}\label{stratK4SparseStage0}
        \State{For time $t/2$ and while the built graph has at most $b/2$ edges, buy any presented edge with one endpoint in $X$ and the other in $V$.}\label{stratK4SparseStage1}
        \State{For time $t/2$ and while the built graph has at most $b$ edges, buy any presented edge which is contained in $N_{B_{t/2}}(x)$ for at least one $x\in X$.}\label{stratK4SparseStage2}
    \end{algorithmic}
    \end{algorithm}

    We note that, by the assumption that $b\leq t$, \eqref{equa:artbound1} implies that $r=\omega(n^{(3k-4)/(2k-2)})=\omega(n^{3/2-1/(2k-2)})$.
    This in turn implies that $|X|=\lceil n^{3k-3}/r^{2k-2}\rceil=o(n)$.
    
    Let $\hat{G}_1$ and $\hat{G}_2$ denote the (random) graphs containing all the edges offered during stages~\ref{stratK4SparseStage1} and~\ref{stratK4SparseStage2}, respectively.
    Consider the coupling of the random graph process given by \cref{lem:coupling} applied with $k=2$ and $t_1=t_2=t/2$.
    More precisely, there exist some $p_1=(1-o(1))t_1/M$, $p_2=(1-o(1))t_2/M$, $p_1'=(1+o(1))t_1/M$ and $p_2'=(1+o(1))t_2/M$, and a coupling of random graphs $(H_1,\hat{G}_1,H_1',H_2,\hat{G}_2,H_2')$, such that $H_i\sim G(n,p_i)$ and $H_i'\sim G(n,p_i')$ for each $i\in[2]$ (where $H_i'$ corresponds to the union $H_i\cup\mybar{H}_i$ in \cref{lem:coupling}), $(H_1,H_1')$ is independent of $(H_2,H_2')$, $H_2'$ is independent of~$\hat{G}_1$, and a.a.s.
    \begin{equation}\label{equa:K4sparse_coupl1}
        H_1\subseteq \hat{G}_1\subseteq H_1'
    \end{equation}
    and 
    \begin{equation}\label{equa:K4sparse_coupl2}
        H_2\setminus H_1'\subseteq \hat{G}_2\subseteq H_2'.
    \end{equation}
    Note, moreover, that if we are to restrict our random graphs to disjoint subsets of edges $E_1,E_2\subseteq\inbinom{[n]}{2}$, we obtain one further independence property: the appearance of each edge in $(H_2\setminus H_1')\cap E_2$ is independent of $(H_1\cap E_1,H_1'\cap E_1)$ (even if $E_2$ is allowed to depend on $(H_1\cap E_1,H_1'\cap E_1)$).
    
    Let $\hat{B}_1$ and $\hat{B}_2$ be the graphs containing all edges purchased during stages~\ref{stratK4SparseStage1} and~\ref{stratK4SparseStage2}, respectively.
    Let $E_1\coloneqq\left\{\{x,v\}:x\in X, v\in V\right\}\subseteq\inbinom{[n]}{2}$, and observe that $\hat{B}_1\subseteq\hat{G}_1\cap E_1$.
    At the end of Stage~\ref{stratK4SparseStage1}, let $E_2\coloneqq\bigcup_{x\in X}\inbinom{N_{\hat{B}_1}(x)}{2}\subseteq\inbinom{[n]}{2}$,\COMMENT{Note that, while $E_1$ is a deterministic set, $E_2$ is a random one. $\checkmark$} and note that $\hat{B}_2\subseteq\hat{G}_2\cap E_2$ and $E_1\cap E_2=\varnothing$ (so we can make use of the independence mentioned above).
    
    We are first going to show that, in fact, a.a.s.\ 
    \begin{equation}\label{equa:K4sparse_coupla}
        \hat{B}_1=\hat{G}_1\cap E_1.
    \end{equation}
    Indeed, this will hold if during Stage~\ref{stratK4SparseStage1} we never purchase $b/2$ edges.
    Note that, for any \mbox{$x\in X$}, since we have that $\mathbb{E}[e_{H_1'}(x,V)]=(1\pm o(1))t/n$, by \cref{lem:chernoff} and the lower bound on~$t$ it follows that $\mathbb{P}[e_{H_1'}(x,V)\geq 2t/n]=\nume^{-\Omega(t/n)}=\nume^{-\omega(n^{1/2-1/(2k-2)})}$.\COMMENT{Note that $\mathbb{E}[e_{H_1'}(x,V)]=(1-o(1))np_1'=(1\pm o(1))t/n$. $\checkmark$}
    Similarly, $\mathbb{P}[e_{H_1}(x,V)\leq t/2n]=\nume^{-\omega(n^{1/2-1/(2k-2)})}$, as $\mathbb{E}[e_{H_1}(x,V)]=(1-o(1))t/n$.\COMMENT{In this case, $\mathbb{E}[e_{H_1}(x,V)]=(1-o(1))np_1=(1-o(1))t/n$. $\checkmark$}
    By a union bound over all $x\in X$ and~\eqref{equa:K4sparse_coupl1}, we conclude that a.a.s.\ for all $x\in X$ we have that
    \begin{equation}\label{equa:K4sparse_e1H}
        \frac{t}{2n}\leq e_{H_1}(x,V)\leq e_{H_1'}(x,V)\leq \frac{2t}{n},
    \end{equation}
    and together with \eqref{equa:K4sparse_coupl1} it follows that a.a.s.\ for every $x\in X$ we have that
    \begin{equation}\label{equa:K4sparse_e1}
        e_{\hat{G}_1}(x,V)\leq\frac{2t}{n}.
    \end{equation}
    In particular, by the upper bound on $t$ and~\eqref{equa:artbound1}, a.a.s.\COMMENT{We are using the upper bound $|X|\leq2n^{3k-3}/r^{2k-2}$, which holds since $r=o(t)=o(n^{3/2})$ (this is why we specify ``the upper bound on $t$''). $\checkmark$}
    \[e_{\hat{G}_1}(X,V)\leq|X|\frac{2t}{n}\leq4\frac{n^{3k-4}t}{r^{2k-2}}=o(b),\] so \eqref{equa:K4sparse_coupla} holds.
    
    We next claim that a.a.s.\ 
    \begin{equation}\label{equa:K4sparse_couplb}
        \hat{B}_2=\hat{G}_2\cap E_2.
    \end{equation}
    This holds if during Stage~\ref{stratK4SparseStage2} we never purchase $b/2$ edges.
    Indeed, suppose that \eqref{equa:K4sparse_e1} holds (which occurs a.a.s.).
    Then, by the upper bound on~$t$, we have that $|E_2|\leq|X|\inbinom{2t/n}{2}\leq 4n^{3k-5}t^2/r^{2k-2}$.
    It follows that $\mathbb{E}[|E(H_2')\cap E_2|]\leq 6n^{3k-7}t^3/r^{2k-2}$ for sufficiently large $n$.
    Since~$H_2'$ is independent of~$\hat{G}_1$, by \cref{lem:chernoff} we have that\COMMENT{We have $\mathbb{E}[|E(H_2')\cap E_2|]\leq 4n^{3k-5}t^2p_2'/r^{2k-2}=(4+o(1))n^{3k-7}t^3/r^{2k-2}$, so $\mathbb{P}\left[|E(H_2')\cap E_2|\geq8n^{3k-7}t^3/r^{2k-2}\right]\leq\nume^{-\Omega(n^{3k-7}t^3/r^{2k-2})}$.
    (Formally, we can see that the variable we care about is stochastically dominated by a binomial random variable $\mathrm{Bin}(4n^{3k-5}t^2/r^{2k-2},p_2')$.) $\checkmark$}
    \[\mathbb{P}\left[|E(H_2')\cap E_2|\geq8\frac{n^{3k-7}t^3}{r^{2k-2}}\right]\leq\exp\left(-\Omega\left(\frac{n^{3k-7}t^3}{r^{2k-2}}\right)\right)=\nume^{-\Omega(n^{1/2})}=o(1),\]
    where the second comparison holds since $r^{2k-2}/t^3=o(t^{2k-5})=o(n^{3k-15/2})$ by the upper bounds on~$r$ and~$t$, respectively.
    Lastly, note that, by the upper bound on $t$, \eqref{equa:artbound1} implies that \[\frac{n^{3k-7}t^3}{r^{2k-2}}=\frac{t^2}{n^3}\frac{n^{3k-4}t}{r^{2k-2}}=o(b),\]
    and thus a.a.s.\ $|E(H_2')\cap E_2|=o(b)$, so \eqref{equa:K4sparse_couplb} follows by~\eqref{equa:K4sparse_coupl2}.\COMMENT{Formally, we are first proving that, in the space conditional on \eqref{equa:K4sparse_e1}, a.a.s.\ we buy few edges (here we are using the fact that $H_2'$ is independent of $(H_1,\hat{G}_1,H_1')$).
    Since \eqref{equa:K4sparse_e1} holds a.a.s., we can undo the conditioning and retain the conclusion.
    Then we take the intersection with the event that \eqref{equa:K4sparse_coupl2} holds, which also occurs a.a.s. $\checkmark$}

    Combining \eqref{equa:K4sparse_coupl1} with \eqref{equa:K4sparse_coupla} and \eqref{equa:K4sparse_coupl2} with \eqref{equa:K4sparse_couplb}, respectively, we deduce that a.a.s.\ $H_1\cap E_1\subseteq\hat{B}_1$ and $(H_2\setminus H_1')\cap E_2\subseteq\hat{B}_2$.
    In order to conclude that the strategy is successful, it thus suffices to verify that a.a.s.\ $(H_1\cap E_1)\cup((H_2\setminus H_1')\cap E_2)$ contains a copy of $W_k$.
    For this, we may first expose $F_1\coloneqq H_1\cap E_1$, and note that the bounds in \eqref{equa:K4sparse_e1H} hold a.a.s. 
    Next, let $E_2^*\coloneqq\bigcup_{x\in X}\inbinom{N_{F_1}(x)}{2}$, and note that a.a.s.\ $E_2^*\subseteq E_2$ by~\eqref{equa:K4sparse_coupl1}.
    We next expose $F_2\coloneqq(H_2\setminus H_1')\cap E_2^*$ (where each edge of $E_2^*$ is retained independently with some probability $p^*=(1-o(1))t/n^2$).\COMMENT{Each edge is retained with probability $p_2(1-p_1')$. $\checkmark$}
    Note that, if for some $x\in X$ the graph $F_2\cap\inbinom{N_{F_1}(x)}{2}$ contains a cycle of length $k-1$, then $F_1\cup F_2$ contains a copy of~$W_k$, so it suffices to prove that the former holds a.a.s.
    
    In order to prove this, we first claim that, for every fixed integer $\ell\geq3$, a.a.s.
    \begin{enumerate}[label=$(\mathrm{CN}\ell)$]
        \item\label{item:common_neighs_ell} every $\ell$-set of vertices in~$\inbinom{V}{\ell}$ is contained in the $F_1$-neighbourhood of at most five vertices $x\in X$.
    \end{enumerate}
    Indeed, recall that, by the assumption that $b\leq t$, \eqref{equa:artbound1} implies that $r=\omega(n^{(3k-4)/(2k-2)})$.
    Now, for a fixed $U\in\inbinom{V}{\ell}$, using the upper bound on~$t$ and this lower bound on~$r$, the probability that~$U$ is contained in the $F_1$-neighbourhood of at least six vertices $x\in X$ is at most\COMMENT{We argue more generally by analysing the probability that they are contained in more than $a_\ell$ such sets.
    The $\ell$ vertices must all be neighbours of at least $a_\ell+1$ vertices in $X$, thus the first term.
    In the second equality we are using the fact that $t=O(n^{3/2})$ and that $r^{2k-2}=\omega(n^{3k-4})$ by the bounds on $t$ and $r$.
    For the last inequality, we must choose some $a_\ell$ such that 
    \[\left(1-\frac{\ell}{2}\right)\left(a_\ell+1\right)\leq-\ell\iff a_\ell\left(1-\frac{\ell}{2}\right)<-1-\ell/2\iff a_\ell\geq\frac{\ell/2+1}{\ell/2-1}=\frac{\ell+2}{\ell-2}.\]
    For reference, we have that $a_3=5$, $a_4=a_5=3$, and $a_\ell=2$ for all $\ell\geq6$. $\checkmark$}
    \[\binom{|X|}{6}(p_1)^{6\ell}=\Theta\left(\left(\frac{n^{3k-3}}{r^{2k-2}}\frac{t^\ell}{n^{2\ell}}\right)^6\right)=o\left(\left(n^{1-\ell/2}\right)^6\right)=o(n^{-\ell}),\]
    and the conclusion follows by a union bound over all $\ell$-sets in $\inbinom{V}{\ell}$.
    
    In a similar fashion, we claim that a.a.s.
    \begin{enumerate}[label=$(\mathrm{CN}2)$]
        \item\label{item:common_neighs_2} every pair of vertices in~$\inbinom{V}{2}$ is contained in the $F_1$-neighbourhood of at most $4k$ vertices $x\in X$.\COMMENT{The value $4k$ is not optimal, and we made no effort to optimise it. $\checkmark$}
    \end{enumerate}
    To show this, fix an arbitrary pair of vertices $U\in\inbinom{V}{2}$.
    We now further split our range for $t$ into two.
    Assume first that $t=O(n^{3/2-1/(4k-4)})$.
    Combining this with \eqref{equa:artbound1} and the assumption that $b\leq t$, with calculations analogous to those for proving \ref{item:common_neighs_ell}, we conclude that the probability that $U$ is contained in the $F_1$-neighbourhood of at least $4k$ vertices is at most
    \[\binom{|X|}{4k}(p_1)^{8k}=\Theta\left(\left(\frac{n^{3k-3}}{r^{2k-2}}\frac{t^2}{n^4}\right)^{4k}\right)=o\left(\left(\frac{tb}{n^3}\right)^{4k}\right)=o\left(\left(\frac{t^2}{n^3}\right)^{4k}\right)=o\left(n^{-\frac{4k}{2k-2}}\right)=o(n^{-2}).\]
    On the other hand, if $t=\omega(n^{3/2-1/(4k-4)})$, by the monotonicity of the ``successfulness'' of $(t,b)$-strategies over~$b$, we may assume that $b\leq n^{3k-4}\cdot n^{k/(4k-4)}/t^{2k-3}$.\COMMENT{We may assume that $b\leq n^{3k-4}\cdot n^{\eps}/t^{2k-3}$ for an arbitrary $\eps>0$, and we make the choice for $\eps$ that we write. $\checkmark$}
    Then, analogously as above, the probability that~$U$ is contained in the $F_1$-neighbourhood of at least $4k$ vertices is at most
    \begin{align*}
        \binom{|X|}{4k}(p_1)^{8k}&=o\left(\left(\frac{tb}{n^3}\right)^{4k}\right)=o\left(\left(\frac{n^{3k-4}}{t^{2k-3}}n^{\frac{k}{4k-4}}\frac{t}{n^3}\right)^{4k}\right)=o\left(\left(\left(\frac{n^3}{t^2}\right)^{k-2}n^{\frac{k}{4k-4}-1}\right)^{4k}\right)\\
        &=o\left(n^{4k\left(\frac{k-2}{2k-2}+\frac{k}{4k-4}-1\right)}\right)=o(n^{-k})=o(n^{-2}).
    \end{align*}
    In both cases, the conclusion follows by a union bound over all possible pairs of vertices.
    
    Condition on the event that \ref{item:common_neighs_ell} holds for all $3\leq\ell\leq k-1$, that \ref{item:common_neighs_2} holds, and that the bounds in \eqref{equa:K4sparse_e1H} hold.
    We now claim that the desired conclusion follows from the second moment method.
    Indeed, let $\Delta$ denote the number of $(k-1)$-sets of vertices $\{v_1,\ldots,v_{k-1}\}\in\inbinom{V}{k-1}$ contained in $N_{F_1}(x)$ for any $x\in X$, and let $Z$ denote the number of cycles of length $k-1$ in $F_2$ whose vertex set is contained in $N_{F_1}(x)$ for some $x\in X$.
    For sufficiently large $n$, it follows from \ref{item:common_neighs_ell} and \eqref{equa:K4sparse_e1H} that\COMMENT{For the constants, in the lower bound we divide by $2$ once more to account for error terms when multiplying the number in the binomial coefficient.
    For the upper bound we add an additional factor of $2$ to account for the rounding in the size of $X$; this leads to multiplying by a factor of $2^k/(k-1)!$, which is in general smaller than $1$ for larger values of $k$, and can be bounded by $4$ for all values of $k$ that we consider. $\checkmark$}
    \begin{equation}\label{equa:Deltabound}
        \frac{1}{5\cdot2^k(k-1)!}\left(\frac{n^2t}{r^2}\right)^{k-1}\leq\frac{|X|}{5}\binom{t/2n}{k-1}\leq\Delta\leq|X|\binom{2t/n}{k-1}\leq4\left(\frac{n^2t}{r^2}\right)^{k-1},
    \end{equation}
    and so $\mathbb{E}[Z]=\Delta(k-2)! (p^*)^{k-1}/2=\Theta((t/r)^{2k-2})=\omega(1)$.
    By Chebyshev's inequality, it now suffices to verify that $\mathrm{Var}(Z)=o((t/r)^{4k-4})$.
    This is by now a standard argument, but we include the details for the interested reader.
    We may express $Z$ as a sum of indicator random variables $Z=\sum_CZ_C$, where the sum is over all copies~$C$ of~$C_{k-1}$ in~$K_n$ which are fully contained in $\inbinom{N_{F_1}(x)}{2}$ for some $x\in X$, where~$Z_C$ is the indicator variable that $C\subseteq F_2$.
    We then have that
    \[\operatorname{Var}(Z)=\sum_{C}\operatorname{Var}(Z_C)+\sum_{C,C'}\operatorname{Cov}(Z_C,Z_{C'})\leq\mathbb{E}[Z]+\sum_{C,C'}\operatorname{Cov}(Z_C,Z_{C'}),\]
    where the sum is now over all ordered pairs $C,C'$ of copies of $C_{k-1}$ as above.
    
    It now remains to bound these covariances.
    The sum can be separated depending on the number of edges shared by $C$ and~$C'$.
    If $C$ and $C'$ share exactly $a\in[k-3]$ edges,\COMMENT{Note that it cannot be that $C$ and $C'$ share exactly $k-2$ edges.} we have that 
    \[\operatorname{Cov}(Z_C,Z_{C'})=(p^*)^{2k-2-a}-(p^*)^{2k-2}\leq(p^*)^{2k-2-a}.\]
    Moreover, the number of pairs of cycles $C,C'$ sharing exactly $a\in[k-3]$ edges is
    \[O\left(\left(\frac{n^2t}{r^2}\right)^{k-1}\left(\frac{t}{n}\right)^{k-a-2}\right);\]
    this can be estimated by first choosing $C$ (for which we first choose some $x\in X$, then some $(k-1)$-set in~$N_{F_1}(x)$, and then a cyclic ordering of the vertices in this set, leading to a similar expression as in the upper bound in \eqref{equa:Deltabound}), then choosing the (at most) $k-a-2$\COMMENT{Note that any two cycles sharing $a<|E(C)|$ edges must share at least $a+1$ vertices.} remaining vertices for $C'$ (which by \ref{item:common_neighs_2} and \ref{item:common_neighs_ell} must be contained in the $F_1$-neighbourhood of one of constantly many vertices $x'\in X$, and for each of which we have $O(t/n)$ choices by \eqref{equa:K4sparse_e1H}), and then fixing their cyclic order (in at most constantly many ways).
    Combining these observations and using the upper bound on $t$, we conclude that\COMMENT{For the first equality, we simply note that $\mathbb{E}[Z]=\Theta((t/r)^{2k-2})$ and so take this term out as a common factor (which removes all the factors of $r$ from the sum); the expression then follows by substituting $p^*=\Theta(t/n^2)$.
    The last inequality holds by the upper bound on $t$.}
    \begin{align*}
        \mathrm{Var}(Z)&\leq\mathbb{E}[Z]+\sum_{a=1}^{k-3}O\left(\left(\frac{n^2t}{r^2}\right)^{k-1}\left(\frac{t}{n}\right)^{k-a-2}(p^*)^{2k-2-a}\right)\\
        &=\mathbb{E}[Z]\left(1+\sum_{a=1}^{k-3}O\left(\frac{t^{2k-2a-3}}{n^{3k-3a-4}}\right)\right)=(1+o(1))\mathbb{E}[Z].
    \end{align*}
    
    \textbf{Case 2.}
    Assume now that $t=\omega(n^{3/2})$.
    Note that in this range we have $b=\omega(n^2/t)$.\COMMENT{Indeed, we have that $\frac{n^{3k-4}/t^{2k-3}}{n^2/t}=n^{3(k-2)}/t^{2(k-2)}=o(1)$ by the lower bound on $t$. $\checkmark$}
    Fix any such~$b$ and let $\tilde{n}=\tilde{n}(n)$ and $t_1=t_1(n)$ be such that $\tilde{n}=o(b)$, $\tilde{n}=o(t/n)$, $\tilde{n}=\omega(n^2/t)$, $\tilde{n}n\leq t_1=o(t)$, and $t_1=\omega(n^2/\tilde{n}^2)$.
    It is straightforward to verify that such functions exist in this range.\COMMENT{Indeed, note that $n^2/t=o(t/n)\iff t=\omega(n^{3/2})$ (which holds), so we can choose such a $\tilde{n}$.
    For $t_1$, we have that $\tilde{n}n=o(t)$, so it suffices to verify that $t=\omega(n^2/\tilde{n}^2)\iff\tilde{n}=\omega(n/\sqrt{t})$, which holds.
    The reason why we make this choice is so that the set where we apply stage~\ref{stratK4DenseStage2}, below, grows in size with $n$, so we can apply the asymptotic results for cycles. $\checkmark$} 
    
    Now consider the $(t,b)$-strategy outlined in \cref{strat:K4dense} below.
    For sufficiently large~$n$, this is a $(t,b)$-strategy by construction, so it only remains to prove that it is successful.

    \begin{algorithm}
    \caption{A $(t,b)$-strategy for $W_k$ for $t=\omega(n^{3/2})$.}\label{strat:K4dense}
    \begin{algorithmic}[1]
        \State{Fix a vertex $x\in[n]$.
        For time $t_1$ and while the built graph has at most $\tilde{n}$ edges, buy any presented edge with one endpoint being $x$.}\label{stratK4DenseStage1}
        \State{For time $t/2$, simulate an optimal $(t\tilde{n}^2/4n^2,b/2)$-strategy for constructing a cycle of length $k-1$ on $N_{B_{t_1}}(x)$.}\label{stratK4DenseStage2}
    \end{algorithmic}
    \end{algorithm}

    Reveal first the set of edges incident to $x$ which are offered during the first $t_1$ steps, and their order, but nothing else of the random graph process.
    This determines $N_{B_{t_1}}(x)$.
    Note that, by our choice of~$\tilde{n}$, a straightforward application of \cref{lem:chernoff} shows that a.a.s.\ $|N_{B_{t_1}}(x)|=\tilde{n}$;\COMMENT{We need to show that a.a.s.\ $|N_{G_{t_1}}(x)|\geq \tilde{n}$.
    Indeed, $|N_{G_{t_1}}(x)|$ follows a hypergeometric distribution with parameters $M$, $n-1$ and $t_1$, so it has expectation $2t_1/n\geq 2\tilde{n}$.
    Since $\tilde{n}=\omega(1)$, the result follows by \cref{lem:chernoff} for all sufficiently large~$n$. $\checkmark$} we will assume that this is the case from now on.

    Now let $t'\coloneqq t\tilde{n}^2/4n^2$ and consider the application of the optimal $(t',b/2)$-strategy for constructing a cycle of length $k-1$ on~$N_{B_{t_1}}(x)$ in Stage~\ref{stratK4DenseStage2} (note that $|N_{B_{t_1}}(x)|=\tilde{n}=\omega(1)$).
    Since we have not revealed the edges of the first stage of the random graph process except those incident to $x$, we have no information about those offered inside $N_{B_{t_1}}(x)$.
    As such, the $(N_{B_{t_1}}(x),t_1,t/2)$-RGP has the same distribution as the segment of length $\ell(N_{B_{t_1}}(x),t_1,t/2)$ of the random graph process on vertex set~$N_{B_{t_1}}(x)$ started at time~$0$.\COMMENT{We might like to claim that the $(N_{B_{t_1}}(x),t_1,t/2)$-RGP has the same distribution as the $(N_{B_{t_1}}(x),0,t/2)$-RGP; however, this is not so clear, as the length of these processes is a random variable and the information we have revealed for analysing Stage~\ref{stratK4DenseStage1} influences the length of the simulated processes. $\checkmark$}
    Moreover, from \cref{lem:simulated_length} we know that a.a.s.~$\ell(N_{B_{t_1}}(x),t_1,t/2)\geq t'$.\COMMENT{One needs to verify that the required conditions for applying this lemma hold. Indeed, we have that $n\tilde{n}\leq t_1=o(n^2)$  since $\tilde{n}=o(t/n)=o(n)$, and $t\tilde{n}^2/2\geq t_1\tilde{n}^2=\omega(n^2)$, where the last equality is true by our choice of $t_1$ and $\tilde{n}$. $\checkmark$}
    Therefore, since $b=\omega(\max\{\tilde{n}^{\lfloor k/2\rfloor+1}/(t')^{\lfloor k/2\rfloor},\tilde{n}/(t')^{1/2}\})$,\COMMENT{Indeed, note first that 
    \[\frac{\tilde{n}^{\lfloor k/2\rfloor+1}}{(t')^{\lfloor k/2\rfloor}}=\Theta\left(\frac{n^{2\lfloor k/2\rfloor}}{t^{\lfloor k/2\rfloor}\tilde{n}^{\lfloor k/2\rfloor-1}}\right)\]
    and that 
    \[\frac{n^2}{t}=\omega\left(\frac{n^{2\lfloor k/2\rfloor}}{t^{\lfloor k/2\rfloor}\tilde{n}^{\lfloor k/2\rfloor-1}}\right)\iff 1=\omega\left(\left(\frac{n^2}{t\tilde{n}}\right)^{\lfloor k/2\rfloor-1}\right)\iff\tilde{n}=\omega(n^2/t),\]
    which holds by the choice of $\tilde{n}$, so the first bound follows since $b=\omega(n^2/t)$.
    For the second bound, simply note that $\tilde{n}/(t')^{1/2}=\Theta(n/t^{1/2})$, which does not tend to $0$ by the upper bound on $t$, and thus $b=\omega(n^2/t)=\omega(n/t^{1/2})=\omega(\tilde{n}/(t')^{1/2})$. $\checkmark$} \cref{lem:CycleStrat} ensures that a.a.s.~Stage~\ref{stratK4DenseStage2} in \cref{strat:K4dense} constructs a cycle of length $k-1$ inside the Stage-\ref{stratK4DenseStage1}-neighbourhood of~$x$, which results in a copy of~$W_k$, as desired.
\end{proof}

In order to prove the upper bound on the optimal budget for $(t,b)$-strategies for constructing a copy of $K_5$ from \cref{thm:K5_main}, we also exhibit an explicit strategy.
This bound is reformulated next, and the proof closely follows that for wheels.

\begin{theorem}\label{thm:K5strategy}
    For all $t=\omega(n^{3/2})$ with $t\leq M$, if 
    \[b=\omega\left(\max\left\{\frac{n^{12}}{t^7},\left(\frac{n^2}{t}\right)^{5/3}\right\}\right),\] 
    there exists a~successful $(t,b)$-strategy for constructing a copy of $K_5$.
\end{theorem}

\begin{proof}
    The proof goes along the same lines as the proof of \cref{thm:WheelStrategy}.
    Due to the analogies between the proofs, we omit some details here.
    We begin by splitting the range of $t$ into two cases, noting that
    \[ \frac{n^{12}}{t^7} = \left(\frac{n^2}t\right)^{5/3} \iff t=n^{13/8} .\]

    \textbf{Case 1.}
    Assume first that $t=O(n^{13/8})$.
    Note that in this range we have $b=\omega(n^{12}/t^7)$.
    Fix any such~$b$ and let $r=r(n)$ be such that $r=o(t)$ but it is sufficiently close to $t$ that
    \[ r=\omega(n^{3/2}), \qquad \frac{n^9t^4}{r^{10}}=o(1), \qquad \text{and} \qquad \frac{n^{12}t^3}{r^{10}}=o(b).\] 
    Then, consider the strategy outlined in \cref{strat:K5sparse} below.
    This is a $(t,b)$-strategy by construction, so it only remains to prove that it is successful for constructing a copy of~$K_5$.

    \begin{algorithm}
    \caption{A $(t,b)$-strategy for $K_5$ for $t=O(n^{13/8})$.}\label{strat:K5sparse}
    \begin{algorithmic}[1]\setcounter{ALG@line}{-1}
        \State{Set $X\coloneqq[\lceil n^{16}/r^{10}\rceil]$ and $V\coloneqq[n]\setminus X$.}\label{stratK5SparseStage0}
        \State{For time $t/2$ and while the built graph has at most $b/2$ edges, buy any presented edge with one endpoint in $X$ and the other in $V$.}\label{stratK5SparseStage1}
        \State{For time $t/2$ and while the built graph has at most $b$ edges, buy any presented edge which is contained in $N_{B_{t/2}}(x)$ for at least one $x\in X$.}\label{stratK5SparseStage2}
    \end{algorithmic}
    \end{algorithm}

First, note that, by the choice of $r$, we have $|X| = \lceil n^{16}/r^{10}\rceil = \lceil n(n^{3/2}/r)^{10}\rceil = o(n)$.
We then define the coupling of random graphs $(H_1,\hat{G}_1,H_1',H_2,\hat{G}_2,H_2')$, the graphs $\hat{B}_1, \hat{B}_2$, and the edge sets $E_1, E_2$ analogously as in the proof of \cref{thm:WheelStrategy}.

Following the proof of \cref{thm:WheelStrategy}, since $t=\omega(n^{3/2})$ and by our choice of~$r$, we note that a.a.s.\ the number of edges purchased during Stage~\ref{stratK5SparseStage1} satisfies
\begin{equation*}\label{eq:K5_stage1_budget_limit}
    e(\hat{B}_1)\leq e_{\hat{G}_1}(X,V)\leq|X|\frac{2t}{n} \leq 4\frac{n^{15}t}{r^{10}} \leq 4\frac{n^{12}t^3}{r^{10}} = o(b),
\end{equation*}
and thus a.a.s.\ $\hat{B}_1=\hat{G}_1\cap E_1$.
Similarly, a.a.s.\ the number of edges purchased during Stage~\ref{stratK5SparseStage2} satisfies
\begin{equation*}\label{eq:K5_stage2_budget_limit}
    e(\hat{B}_2)\leq |E(H_2')\cap E_2|\leq|X|\binom{2t/n}{2}\frac{2t}{n^2} \leq 8\frac{n^{12}t^3}{r^{10}} = o(b),
\end{equation*}
and thus a.a.s.\ $\hat{B}_2=\hat{G}_2\cap E_2$.
Therefore, after defining $F_1$ and $F_2$ analogously as in the proof of \cref{thm:WheelStrategy}, in order to verify that the strategy is successful, %(in other words, that during stage~\ref{stratK5SparseStage2} Builder will buy at least one copy of $K_4$ lying inside $N_{\hat{B}_1}(x)$ for some $x\in X$), 
it suffices to prove that a.a.s.\ there is some $x\in X$ such that $F_2\cap\inbinom{N_{F_1}(x)}{2}$ contains a copy of $K_4$ (recall that, after revealing $F_1$, each edge in $\inbinom{N_{F_1}(x)}{2}$ appears in $F_2$ independently with some probability $p^*=(1-o(1))t/n^2$).

To prove this, we first claim that for every $\ell\geq3$ we have that a.a.s.\ 
\begin{enumerate}[label=$(\mathrm{CN}\ell)$]
    \item\label{item:common_neighs_ellK5} every $\ell$-set of vertices of $V$ is contained in the $F_1$-neighbourhood of at most five vertices $x\in X$.
\end{enumerate}    
Indeed, for a fixed $U\in\inbinom{V}{\ell}$, by our choice of $r$ and the bounds on $t$, the probability that there are at least six vertices $x\in X$ such that $U\subseteq N_{F_1}(x)$ is at most\COMMENT{To see the last inequality, observe first that in the case $\ell=3$ we have $n/t=o(n^{-1/2})$ by the lower bound on $t$, and thus $(n/t)^6=o(n^{-3})$.
For each subsequent value of $\ell$, we are multiplying by a factor of $(t/n^2)^6=O(n^{-6\cdot3/8})=o(n^{-1})$, so the inequality holds inductively. 
$\checkmark$}
\[ \Theta\left(\binom{|X|}{6} \left(\frac{t}{n^2}\right)^{6\ell}\right) = \Theta\left(\left(\frac{n^{16}}{r^{10}}\frac{t^\ell}{n^{2\ell}}\right)^6 \right) = \Theta\left(\left(t^{\ell-4}n^{7-2\ell}\frac{n^9t^4}{r^{10}} \right)^6 \right) = o\left(\left(t^{\ell-4}n^{7-2\ell}\right)^6\right) = o(n^{-\ell}).\]
Hence, by a union bound, a.a.s.\ for every $\ell$-set $U\in\inbinom{V}{\ell}$, there are at most five vertices $x\in X$ such that $U\subseteq N_{F_1}(x)$. 

Now condition on the event that \ref{item:common_neighs_ellK5} holds for $\ell\in\{3,4\}$ and that the bounds in \eqref{equa:K4sparse_e1H} hold (both of which hold a.a.s.).
Let $Z$ denote the number of $4$-sets in $\inbinom{V}{4}$ that induce a $K_4$ in $F_2$ and are contained in $N_{F_1}(x)$ for some $x\in X$.
Then, again by our choice of $r$,
\[ \mathbb{E}[Z] = \Theta\left(|X|\left(\frac{t}{n}\right)^4\left(\frac{t}{n^2}\right)^6 \right) = \Theta\left(\left(\frac{t}{r}\right)^{10}\right) = \omega(1).\]
It now remains to bound the variance, similarly as in the proof of \cref{thm:WheelStrategy}.
We have to consider the cases where two different potential copies of $K_4$ share two or three vertices (for which the covariances of the corresponding indicator random variables are bounded by~$(p^*)^{11}$ and~$(p^*)^9$, respectively).
For bounding the number of pairs of indicator variables which we consider, we first make a choice of four vertices within the $F_1$-neighbourhood of some $x\in X$, for which there are $O(|X|(t/n)^4)$ choices by the analogue of \eqref{equa:K4sparse_e1H}.
Next we bound the number of copies of $K_4$ which intersect the first copy we fixed.
In the second case (that is, the case that they intersect in three vertices), once we choose which three vertices are shared, by~\ref{item:common_neighs_ellK5} there are at most five choices for a vertex $x'\in X$ such that the second copy of~$K_4$ is contained in the neighbourhood of~$x'$; once such an~$x'$ is fixed, there are at most $2t/n$ choices for the fourth vertex, again by the corresponding analogue of~\eqref{equa:K4sparse_e1H}.
In the first case, however, once we fix the two shared vertices, we do not have any immediate restrictions on the other two.
Thus, we first fix a (potentially arbitrary) third vertex, for which there are at most~$n$ choices.
After this has been fixed, by~\ref{item:common_neighs_ellK5} there are at most five choices for an $x''\in X$ such that these three vertices (the two shared and the third one we chose) are contained in the neighbourhood of~$x''$.
Then we can choose the fourth vertex within this neighbourhood.
Combining these observations and using also the upper bound on~$t$, we conclude that
\begin{align*}
    \mathrm{Var}(Z) & \leq \mathbb{E}[Z] + O\left(|X|\left(\frac{t}{n}\right)^4\left(n\frac{t}{n}\left(\frac{t}{n^2}\right)^{11} + \frac{t}{n}\left(\frac{t}{n^2}\right)^{9}\right) \right)\\
    & = \mathbb{E}[Z] + O\left(|X|\left(\frac{t}{n}\right)^4\left(\frac{t}{n^2}\right)^{6}\left(\frac{t^6}{n^{10}} + \frac{t^4}{n^7}\right) \right) = (1+o(1))\mathbb{E}[Z].
\end{align*}
Hence, by Chebyshev's inequality, a.a.s.\ during Stage~\ref{stratK5SparseStage2} Builder will claim at least one copy of $K_4$ which completes a copy of $K_5$.

\textbf{Case 2.}
Assume now that $t=\omega(n^{13/8})$.
Note that in this range we have $b = \omega((n^2/t)^{5/3})$.
Fix any such~$b$ and let $\tilde{n}=\tilde{n}(n)$ and $t_1=t_1(n)$ be such that $\tilde{n}=o(b)$, $\tilde{n}=o(t/n)$, $\tilde{n}=\omega((n^2/t)^{5/3})$, $\tilde{n}n\leq t_1=o(t)$, and $t_1=\omega(n^2/\tilde{n}^2)$.\COMMENT{Again, one needs to verify that such a choice is possible. Note that $(n^2/t)^{5/3}=o(t/n)\iff t=\omega(n^{13/8})$, which is our regime, so we can choose such $\tilde{n}$.
For $t_1$, we have that $\tilde{n}n=o(t)$, so it suffices to verify that $t=\omega(n^2/\tilde{n}^2)\iff\tilde{n}=\omega((n^2/t)^{1/2})$, which holds as $n^2>t$ and $\tilde{n}=\omega((n^2/t)^{5/3})$. $\checkmark$}
Now consider the $(t,b)$-strategy outlined in \cref{strat:K5dense} below.
For sufficiently large~$n$, this is a $(t,b)$-strategy by construction, so it only remains to prove that it is successful.

    \begin{algorithm}
    \caption{A $(t,b)$-strategy for $K_5$ for $t=\omega(n^{13/8})$.}\label{strat:K5dense}
    \begin{algorithmic}[1]
        \State{Fix a vertex $x\in[n]$.
        For time $t_1$ and while the built graph has at most $\tilde{n}$ edges, buy any presented edge with one endpoint being $x$.}\label{stratK5DenseStage1}
        \State{For time $t/2$, simulate an optimal $(t\tilde{n}^2/4n^2,b/2)$-strategy for constructing a copy of $K_4$ on $N_{B_{t_1}}(x)$.}\label{stratK5DenseStage2}
    \end{algorithmic}
    \end{algorithm} 
    
We can use the same reasoning as in the proof of Case~2 of \cref{thm:WheelStrategy} with $t'\coloneqq t\tilde{n}^2/4n^2$.
Note that by our choice of~$\tilde{n}$ we have that $b = \omega(\max\{\tilde{n}^8/(t')^5, \tilde{n}^2/t'\})$\COMMENT{Indeed, note that $\tilde{n}^8/(t')^5 = 4^5n^{10}/t^5\tilde{n}^2 = o((n^2/t)^{5/3}) \iff \tilde{n} = \omega((n^2/t)^{5/3})$, which holds by our assumption on~$\tilde{n}$. Thus $\tilde{n}^8/(t')^5 = o(b)$.
For the second condition, we have $\tilde{n}^2/t' = 4n^2/t = O((n^2/t)^{5/3}) =o(b)$. $\checkmark$}.
Thus, by \cref{thm:WheelStrategy} with $k=4$, we conclude that a.a.s.\ Stage~\ref{stratK5DenseStage2} in Strategy~\ref{strat:K5dense} constructs a copy of $K_4$ inside $N_{B_{t_1}}(x)$, which results in a copy of $K_5$.
\end{proof}

\section{Concluding remarks and open problems}\label{sect:conclusions}

\subsection{Lower bounds}

\Cref{lem:max_no_copies} offers a general approach to derive lower bounds for the ``budget threshold'' for constructing copies of any given fixed graph.
With this general lemma we are able to recover previous results from the literature, such as the correct lower bounds for constructing triangles and arbitrary trees (first obtained by \citet{FKM25}), as well as diamonds and $k$-fans (first obtained by \citet{ILS24}).
Moreover, we also obtain several general new results, such as the correct lower bound for wheels on $k$ vertices (\cref{cor:Wheel_lower_bound}) and copies of~$K_{1,T}$ for arbitrary trees~$T$ (\cref{thm:KT_lower}).
%We know, however, that the result is far from tight in general; in particular, we cannot recover the results of \citet{FKM25} for cycles of any length at least $4$.

However, the results which follow from \cref{lem:max_no_copies} are not tight in general, and one can readily verify that we do not recover the tight lower bound for the optimal budget for constructing a copy of any cycle of length at least $4$ proved by \citet{FKM25} (see \cref{lem:CycleStrat}).
It is natural to wonder for which other graphs it is possible to obtain a tight lower bound on the ``budget threshold'' by applying \cref{lem:max_no_copies}.

\begin{problem}\label{problem:characterise}
    Characterise all fixed graphs $F$ for which \cref{lem:max_no_copies} yields a tight asymptotic lower bound on the optimal budget $b$ (as a function of~$n$ and~$t$) for which there exist successful $(t,b)$-strategies for constructing a copy of~$F$.
\end{problem}

We emphasise here that, in order to prove our lower bounds on the optimal budget under which there exists a successful $(t,b)$-strategy for constructing some copy of a fixed graph~$F$, in \cref{lem:max_no_copies} we instead obtain upper bounds on the \emph{number of copies} of $F$ that Builder can construct 
(the desired lower bounds follow by simply checking where this upper bound on the number of copies is~$o(1)$, which implies that a.a.s.\ Builder cannot construct any copy of~$F$).
From this point of view, the fact that our results are not tight in general is not too surprising: if we consider the intuition for the counting result that we presented in \cref{sect:lower_bounds} when $np\geq b=\omega(1)$, it is clear that one cannot really hope to always have~$b$ choices for each subsequent vertex that is added towards counting a copy of~$F$.
Indeed, imagine for example that~$F$ contains some edge $e=xy$ and some vertex $z$ at distance at least~$2$ from $e$ (say, at distance $2$ from $x$), and that we start constructing $F$ from some choice for~$e$.
If we first assume that we have $b$ choices for a neighbour $u$ of~$x$, then it is not possible to have $b$ choices for $z$ for each choice of $u$, as this would require a total of $\Theta(b^2)=\omega(b)$ edges.
Similarly, if $np=o(b)$, then one cannot really hope to always have~$np$ choices for each subsequent vertex that is added towards counting a copy of~$F$, as most vertices cannot have such a high degree due to the budget constraint.
Therefore, in order to obtain sharper counting results, a finer approach is needed.

This obvious gap in our approach leads to two interesting problems.
On the one hand, since our counting result cannot be tight in general for all graphs $F$, obtaining general sharp counting results becomes a natural challenge.
Even when \cref{lem:max_no_copies} can be used to obtain a sharp lower bound for the optimal budget for constructing some copy of $F$, it is not clear that our bound on the number of copies of $F$ is tight for all pairs $(t,b)$ of time and budget constraints.
While we have made no effort to study this direction, we believe that a general solution would shed much light on the behaviour of the budget-constrained random graph process.

\begin{problem}\label{problem:counting}
    For any $t=t(n)\in[M]$ and $b=b(n)\in[t]$ and any fixed graph $F$, determine sharp asymptotic bounds for the number of copies of $F$ that $B_t$ may contain when running the budget-constrained random graph process under any $(t,b)$-strategy.
\end{problem}

On the other hand, there may well be other general tools that allow us to obtain a tight lower bound on the optimal budget constraint for general graphs $F$.
A different approach for addressing this problem (which is in line with the work of \citet{ILS24} and with \cref{lem:lower_bound_probability}) is to obtain bounds for a different key quantity: rather than the number of copies of $F$, one can consider the number of pairs of vertices for which there exists a rooted copy of $F\setminus\{e\}$ (a ``trap'').
If this quantity can be estimated, then one can also estimate the probability that a trap is offered throughout the random graph process, which would then complete a copy of $F$.

\subsection{Upper bounds}

In \cref{thm:WheelStrategy,thm:KT_upper} we have provided explicit strategies for Builder to a.a.s.\ construct a graph which contains a copy of a wheel (in particular, a $K_4$) or a $K_{1,T}$ for an arbitrary tree~$T$.
These strategies work with budgets which are optimal up to constant factors for every possible value of~$t$.
Moreover, in \cref{thm:K5strategy} we provided a strategy for constructing a copy of $K_5$ which we believe to be optimal up to a constant factor (see \cref{conj:K5}).
These different strategies are actually very similar, and are particular examples of a general class of strategies.
We believe that the optimal strategies for constructing cliques of any size must be of the same type as those we have considered here.
Roughly speaking, these strategies work as follows.
First, one iteratively grows stars within the set of vertices chosen as leaves of the previous set of stars, until a certain depth.
Then, one attempts to find a copy of a smaller clique contained in one of the sets of leaves of a star built at the previous stage.

More formally, such strategies can be described as follows.
Suppose we want to construct a graph containing a copy of $K_s$, for some $s\geq3$.
We may describe a strategy of ``depth'' $i$ for any $i\in[s]$.
First, fix an arbitrary equipartition of $[n]$ into $s$ sets, $[n]=U_1\cupdot\ldots\cupdot U_s$.
(This plays no real role in the strategy and incurs in constant-factor losses on the required budget for the strategy, but should simplify formalising a proof of its correctness.)
Then, take a random subset $A\subseteq[n]$ of size $\alpha$, for an appropriate choice of $\alpha$, and for each $i\in[s]$ let $A_i\coloneqq A\cap U_i$ and $A_i^+\coloneqq A\cap(\bigcup_{j\in[s]\setminus[i]}U_j)$.
Now fix some $i\in\{0\}\cup[s-1]$.
We are going to discuss a strategy of depth~$i$.

First, we iteratively define a set of vectors of vertices, which simply represent sets of vertices which form a clique in the built graph.
For each $j\in\{0\}\cup[i]$, we shall denote the set of all such vectors of length $j$ by $V_j$.
We begin by setting $V_0\coloneqq\{\varnothing\}$, and $N_0(\varnothing)\coloneqq A$.
Now, for each $j\in[i]$ and assuming that $V_{j-1}$ is defined and $N_{j-1}(\vect{x})$ is defined for every $\vect{x}\in V_{j-1}$, we proceed as follows.
We fix an appropriate integer $k_j$ and, for each $\vect{x}\in V_{j-1}$, we choose a set $Y_j^{\vect{x}}\subseteq N_{j-1}(\vect{x})\cap A_j$ of size $k_j$ uniformly at random.
Then, for time $t/s$ and while at most $b/s$ edges have been bought, we purchase every presented edge with one endpoint in $Y_j^{\vect{x}}$ and the other in $N_{j-1}(\vect{x})\cap A_j^+$, for any $\vect{x}\in V_{j-1}$.
At this point, we set $V_j\coloneqq\{\vect{x}y:\vect{x}\in V_{j-1},y\in Y_j^{\vect{x}}\}$ and, for each $\vect{x}\in V_j$, we let $N_j(\vect{x})$ denote the set of all vertices $v\in A_j^+$ such that the edge $x_jv$ was purchased in this round of exposure.
Once this process has been run for all $j\in[i]$, note that for every $\vect{x}\in V_i$ we have that all edges between vertices of $\vect{x}$ have been purchased throughout the process, and that all vertices in $N_i(\vect{x})$ are neighbours of every vertex of $\vect{x}$ in the built graph.
In  one last round of exposure, for time $t/s$ and while at most $b/s$ edges have been bought, Builder purchases every offered edge which lies in $\bigcup_{\vect{x}\in V_i}\inbinom{N_i(\vect{x})}{2}$.
If this last round of exposure results in a copy of $K_{s-i}$ contained in some $N_i(\vect{x})$, then this clique together with the vertices of $\vect{x}$ forms a copy of $K_s$.

We remark that, for each $i$, the strategy of depth $i$ may succeed only in some range of $t$, and different depths lead to different bounds on the optimal budget; the optimal budget for some $t$ would then be taken as the minimum over all these strategies which are successful for this $t$.
As there are multiple parameters to consider, the analysis of these strategies becomes more cumbersome as $s$ grows, and already for~$K_6$ the behaviour seems to be much more complex.

We have performed a preliminary analysis of the $K_6$ case and represent some of our findings in \cref{fig:K6}; this analysis does not lead to any particular insights and, while it follows the same approach as the proof of \cref{thm:WheelStrategy}, it is much more technical than the ones we have presented, so we have chosen to omit the details.
Still, we would like to comment on some key aspects.
For instance, the strategy of depth $0$ (that is, simply purchasing all edges which are offered within a subset of $[n]$ and hoping that this results in a copy of $K_s$) uses suboptimal budget for all $t$ when $s\in\{3,4,5\}$, but when $s=6$ there is a range of $t$ where this strategy seems to outperform all deeper strategies.
In particular, if optimal, this makes the $\log_nb$-by-$\log_nt$ depiction of optimal budgets as shown in Figures~\ref{fig:wheels}, \ref{fig:K4} and \ref{fig:K6} no longer convex, which is a different behaviour from previously known examples.
Moreover, rather than having two ranges where the behaviour is different, these strategies for $K_6$ seem to yield four distinct ranges, which is again different from all previous results.
In particular, our analysis suggests that, if 
\[M\geq t\geq b=\omega\left(\min\left\{\left(\frac{n^2}{t}\right)^4,\max\left\{\frac{n^{21}}{t^{12}},\frac{n^{16}}{t^9},\left(\frac{n^{2}}{t}\right)^{7/3}\right\}\right\}\right),\]
then there is a successful $(t,b)$-strategy for constructing a copy of $K_6$.
It would be very interesting to confirm whether these strategies are indeed optimal or can be outperformed by different ones.

\begin{figure}
\begin{tikzpicture}[scale=1]
    \begin{axis}[
    %unit vector ratio*=1 1 1,
    axis lines = middle,
    %x=5cm,
    %y=5cm,
    xlabel={$\log_{n}t$},
    ylabel={$\log_{n}b$},
    label style = {below left},
    legend style={at={(1.3,0.75)}},
    xmin=8/5, xmax=2.1, 
    xtick={8/5+0.001,13/8,5/3,1.7,2},
    xticklabels={$8/5$,, $5/3$, $\ \ \ \ \ \ \ \ \ \ 17/10$, $2$},
    ymin=0, ymax=1.9,
    ytick={0.001,2/3,0.7,1,3/2,8/5},
    yticklabels={$0$, $2/3$, , $1$, $3/2$, $8/5$}
    ]
        \addplot[blue, ultra thick, domain=8/5:13/8] {8-4*x};
        \addplot[blue, ultra thick, domain=13/8:5/3] {21-12*x};
        \addplot[blue, ultra thick, domain=5/3:17/10] {16-9*x};
        \addplot[blue, ultra thick, domain=17/10:2] {14/3-7*x/3};
        \addplot[blue, ultra thick, dashed, domain=8/5:5/3] {24-14*x};
        \addplot[blue, ultra thick, dashed, domain=5/3:2] {4-2*x};
        \addplot[gray, thick, dashdotted, domain=8/5:2] {x-1};
        \addplot[gray, thick, dashed, domain=8/5:2] {8-4*x};
        \addplot[gray, dotted, domain=8/5:5/3] {2/3};
        \addplot[gray, dotted, domain=8/5:1.7] {0.7};
        \addplot[gray, dotted, domain=8/5:5/3] {1};
        \addplot[gray, dotted, domain=8/5:13/8] {3/2};
        \addplot[gray, dotted] coordinates {(13/8, 0) (13/8, 3/2)};
        \addplot[gray, dotted] coordinates {(5/3, 0) (5/3, 1)};
        \addplot[gray, dotted] coordinates {(1.7, 0) (1.7, 0.7)};
        \legend{$K_6$ upper bound,,,,$K_6$ lower bound,,$x\mapsto x-1$,$x\mapsto 8-4x$}
    \end{axis}
\end{tikzpicture}
\caption{A depiction of the upper bound and lower bound on the optimal budget~$b$ for successful $(t,b)$-strategies for~$K_6$.
The lower bound follows from \cref{lem:max_no_copies}, and the proof of the upper bound is omitted.}
\label{fig:K6}
\end{figure}
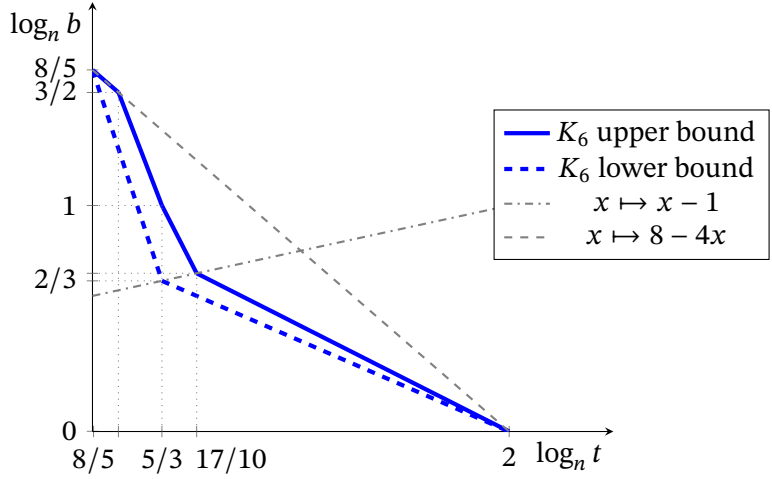

    % \begin{algorithm}
    % \caption{A family of $(t,b)$-strategies for constructing $K_s$.}\label{strat:KsGeneral}
    % \begin{algorithmic}[1]\setcounter{ALG@line}{-1}
    %     \State{Fix an arbitrary equipartition of $[n]$ into $s-1$ sets, $[n]=U_1\cupdot\ldots\cupdot U_{s-1}$.
    %     Fix a set $X_0\subseteq U_1$ of some appropriate size.}\label{stratKsGeneralStage0}
    %     \State{For time $t/2$ and while the built graph has at most $b/2$ edges, buy any presented edge with one endpoint in $X$ and the other in $V$.}\label{stratKsGeneralStage1}
    %     \State{For time $t/2$ and while the built graph has at most $b$ edges, buy any presented edge which is contained in $N_{B_{t/2}}(x)$ for some $x\in X$.}\label{stratKsGeneralStage2}
    % \end{algorithmic}
    % \end{algorithm}

In even more generality, we may think of strategies for other fixed graphs.
A property that cliques (as well as wheels $W_k$ and graphs $K_{1,T}$) satisfy is that they contain \defi{universal} vertices, that is, vertices which are neighbours to every other vertex of the graph.
We believe that, for graphs containing a universal vertex, optimal $(t,b)$-strategies may to an extent have a similar structure to that discussed above, at least for some range of $t$.
More concretely, we wonder whether it is true that, for every fixed graph $F$ containing a universal vertex $v$, there exists some range of $t$ for which the ``budget-threshold'' is attained by a strategy which simply fixes a~vertex $x\in[n]$, spends some amount of time building its neighbourhood, and then the rest of the time simulating an optimal strategy for constructing a copy of $F-v$ within this neighbourhood.
Note that, when $t$ is sufficiently large, this is the case for all the proofs we have presented here.
In even greater generality, we say that a~set of vertices~$S$ is \defi{universal} for a graph $F$ if $S\subseteq V(F)$, $F[S]$ is a (possibly empty) matching, and $S\subseteq N(v)$ for all $v\in V(F)\setminus S$.
We extend the previous question to the setting of graphs containing a universal set.
That is, given a fixed graph $F$ containing a universal set $S$, does there exist some range of $t$ for which the following strategy attains the ``budget-threshold''?
For as long as needed (a.a.s.\ constantly many steps), purchase any set of (vertex-disjoint) offered edges until we have a matching $M$ of size $|E(F[S])|$.
Then, fix an arbitrary set $X\subseteq[n]\setminus V(M)$ of size $|S|-2|E(F[S])|$.
For time $t/2$, purchase any offered edge which is incident to any vertex in~$X\cup V(M)$, up to some bound on the total number of purchased edges.
Then, for time $t/2$, simulate an optimal strategy for constructing a copy of $F-S$ in the (purchased) common neighbourhood of all fixed vertices.

% Use with natbib, not biblatex:
\bibliographystyle{mystyle}
\bibliography{bib}

% Use with biblatex, not natbib:
% \printbibliography

\appendix

\section{Proof of \texorpdfstring{\cref{thm:KT_main}}{Theorem 1.2}}\label{appendix}

As we did with the other main results in our paper, \cref{thm:KT_main} can be split into two statements, one for the lower bound and one for the upper bound on the optimal budget for which we can find a~successful $(t,b)$-strategy.
We begin with the statement for the lower bound.

\begin{theorem}\label{thm:KT_lower}
    Let $T$ be a fixed tree.
    Let $m\coloneqq|E(T)|$.
    For all $t\in[M]$, if
    \begin{equation*}
        t=o\left(n^{\frac{3m}{2m+1}}\right)\qquad\text{ or }\qquad b=o\left(\max\left\{\frac{n^{3m}}{t^{2m}},\left(\frac{n^{2}}{t}\right)^{\frac{m}{m+1}}\right\}\right),
    \end{equation*}
    then for any $(t,b)$-strategy a.a.s.~$B_t$ does not contain a copy of\/ $K_{1,T}$.
\end{theorem}

\begin{proof}
    Observe that $K_{1,T}$ is a graph with $m+2$ vertices and $2m+1$ edges.
    We argue as in the proof of \cref{cor:Wheel_lower_bound}.
    Let $p\coloneqq t/M$.
    If $t=o(n^{3m/(2m+1)})$, then a.a.s.\ $G_t$ contains no copy of~$K_{1,T}$, by the first moment method, %classical results of \citet{ER60},\COMMENT{Indeed, note that $K_{1,T}$ is a strictly balanced graph and that the expected number of copies of $K_{1,T}$ in $G(n,p)$ is $\Theta(n^{m+2}p^{2m+1})$, so the threshold is of order $p=\Theta(n^{-(m+2)/(2m+1)})$, which corresponds to $t=\Theta(n^{2-(m+2)/(2m+1)})$. Finally, notice that $2-\frac{m+2}{2m+1}=2-\frac{m+1/2+3/2}{2m+1}=\frac{3}{2}-\frac{3/2}{2m+1}=\frac{3}{2}\left(1-\frac{1}{2m+1}\right)=\frac{3}{2}\frac{2m}{2m+1}=\frac{3m}{2m+1}.$  $\checkmark$}
    and thus neither does $B_t$.
    Hence, we may assume that $t=\Omega(n^{3m/(2m+1)})$.
    Given an arbitrary $(t,b)$-strategy~$\mathcal{S}$, \cref{lem:max_no_copies} ensures that a.a.s.
    \[\nc(K_{1,T})=\nc(K_{1,T},\cS,n,t,b)\leq\eta\cdot b\cdot\min\left\{b,np\right\}^mp^{m},\]
    where $\eta=\omega(1)$ is a function that grows arbitrarily slowly.
    Observe that
    \[\frac{n^{3m}}{t^{2m}}=\left(\frac{n^{2}}{t}\right)^{\frac{m}{m+1}}\iff t=n^{\frac{3m+1}{2m+1}}.\]
    Suppose first that $t=O(n^{\frac{3m+1}{2m+1}})$.
    For any $b=o(n^{3m}/t^{2m})$, choose some $\eta=\omega(1)$ with $\eta=o(n^{3m}/bt^{2m})$. % (which exists by the choice of $b$).
    Then, we conclude that a.a.s.\COMMENT{We are using the trivial fact that $\min\{b,t/n\}\leq t/n$.}
    \[\nc(K_{1,T})\leq\eta\cdot b\left(np\right)^{m}p^m=O\left(\eta\cdot\frac{bt^{2m}}{n^{3m}}\right)=o(1).\]
    Suppose next that $t=\omega(n^{\frac{3m+1}{2m+1}})$.
    Given any $b=o((n^2/t)^{m/(m+1)})$, choose some $\eta=\omega(1)$ with \mbox{$\eta=o(n^{2m}/b^{m+1}t^{m})$}.
    We then conclude that a.a.s.\COMMENT{We now use the trivial fact that $\min\{b,t/n\}\leq b$.}
    \[\nc(K_{1,T})\leq\eta\cdot b^{m+1}p^m=O\left(\eta\cdot \frac{b^{m+1}t^{m}}{n^{2m}}\right)=o(1).\qedhere\]
\end{proof}

For the upper bound, we proceed analogously as in the proof of \cref{thm:WheelStrategy}.
The most relevant change is that, rather than appealing to \cref{lem:CycleStrat} for the second case of the analysis, we instead use the following result.

\begin{lemma}[{\citet[Theorem~1.5]{FKM25}}]\label{lem:TreeStrat}
    Let $m\geq1$ be a fixed integer and~$T$ be a tree with~$m$ edges.
    If
    \[M\geq t\geq b=\omega(\max\{(n/t)^{m-1},1\}),\]
    then there exists a successful $(t,b)$-strategy for constructing a copy of\/ $T$.
\end{lemma}

\begin{theorem}\label{thm:KT_upper}
    Let $T$ be a fixed non-trivial tree.
    Let $m\coloneqq|E(T)|$.
    If
    \begin{equation*}
        M\geq t\geq b=\omega\left(\max\left\{\frac{n^{3m}}{t^{2m}},\left(\frac{n^{2}}{t}\right)^{\frac{m}{m+1}}\right\}\right),
    \end{equation*}
    then there exists a successful $(t,b)$-strategy for constructing a copy of\/ $K_{1,T}$.
\end{theorem}

\begin{proof}
    We argue like in the proof of \cref{thm:WheelStrategy}.
    By reordering the bounds from the statement, we have that $t=\omega(n^{3m/(2m+1)})=\omega(n)$.
    We consider two cases, depending on the range of~$t$.
    
    \textbf{Case 1.}
    Assume first that $t=O(n^{\frac{3m+1}{2m+1}})$.
    In this range, we have that $b=\omega(n^{3m}/t^{2m})$.
    Fix any such $b$ and let $r=r(n)$ be such that $r=o(t)$ but it is sufficiently close to $t$ that
    \[ r=\omega(n^{3m/(2m+1)}), \qquad \frac{n^{3m}t}{r^{2m+1}}=o(b) \qquad \text{and} \qquad \frac{n^{3m-3}t^3}{r^{2m+1}}=o(b).\] 
    Then, consider the strategy outlined in \cref{strat:KTsparse} below.
    Since this is a $(t,b)$-strategy by construction, it only remains to prove that it is successful for constructing a copy of $K_{1,T}$.

    \begin{algorithm}
    \caption{A $(t,b)$-strategy for $K_{1,T}$ for $t=O(n^{\frac{3m+1}{2m+1}})$.}\label{strat:KTsparse}
    \begin{algorithmic}[1]\setcounter{ALG@line}{-1}
        \State{Set $X\coloneqq[\lceil n^{3m+1}/r^{2m+1}\rceil]$ and $V\coloneqq[n]\setminus X$.}\label{stratKTSparseStage0}
        \State{For time $t/2$ and while the built graph has at most $b/2$ edges, buy any presented edge with one endpoint in $X$ and the other in $V$.}\label{stratKTSparseStage1}
        \State{For time $t/2$ and while the built graph has at most $b$ edges, buy any presented edge which is contained in $N_{B_{t/2}}(x)$ for at least one $x\in X$.}\label{stratKTSparseStage2}
    \end{algorithmic}
    \end{algorithm}
    
    First, note that, by the choice of $r$, we have $|X| = \lceil n^{3m+1}/r^{2m+1}\rceil = o(n)$.
    We then define the coupling of random graphs $(H_1,\hat{G}_1,H_1',H_2,\hat{G}_2,H_2')$, the graphs $\hat{B}_1, \hat{B}_2$, and the edge sets $E_1, E_2$ analogously as in the proof of \cref{thm:WheelStrategy}.
    
    Following the proof of \cref{thm:WheelStrategy}, by the upper bound on $t$ and our choice of~$r$, we note that a.a.s.\ the number of edges purchased during the first stage satisfies
    \[
        e(\hat{B}_1)\leq e_{\hat{G}_1}(X,V)\leq|X|\frac{2t}{n} \leq 4\frac{n^{3m+1}t}{r^{2m+1}n}=o(b),
    \]
    and thus a.a.s.\ $\hat{B}_1=\hat{G}_1\cap E_1$.
    Similarly, a.a.s.\ the number of edges purchased during the second stage satisfies
    \[
        e(\hat{B}_2)\leq |E(H_2')\cap E_2|\leq|X|\binom{2t/n}{2}\frac{2t}{n^2} \leq 8\frac{n^{3m+1}t^3}{r^{2m+1}n^4} = o(b),
    \]
    and thus a.a.s.\ $\hat{B}_2=\hat{G}_2\cap E_2$.
    Therefore, after defining $F_1$ and $F_2$ analogously as in the proof of \cref{thm:WheelStrategy}, in order to verify that the strategy is successful, %(in other words, that during stage~\ref{stratK5SparseStage2} Builder will buy at least one copy of $K_4$ lying inside $N_{\hat{B}_1}(x)$ for some $x\in X$), 
    it suffices to prove that a.a.s.\ there is some $x\in X$ such that $F_2\cap\inbinom{N_{F_1}(x)}{2}$ contains a copy of $T$ (recall that, after revealing $F_1$, each edge in $\inbinom{N_{F_1}(x)}{2}$ appears in $F_2$ independently with some probability $p^*=(1-o(1))t/n^2$).
    
    To prove this, we first claim that for every $\ell\geq3$ we have that a.a.s.
    \begin{enumerate}[label=$(\mathrm{CN}\ell)$]
        \item\label{item:common_neighs_ellKT} every $\ell$-set of vertices of $V$ is contained in the $F_1$-neighbourhood of at most five vertices $x\in X$.
    \end{enumerate}
    Indeed, for a fixed $U\in\inbinom{V}{\ell}$, by our lower bound on $r$ and the bounds on $t$, the probability that there are at least six vertices $x\in X$ such that $U\subseteq N_{F_1}(x)$ is at most\COMMENT{In the first equality, we simply substitute the value of $|X|$. In the second, we use the upper bound on $t$ and the lower bound on $r$ to see that \[ \frac{n^{3m+1}}{r^{2m+1}}\frac{t^\ell}{n^{2\ell}} = o\left(n^{3m+1+\ell\frac{3m+1}{2m+1}-(2m+1)\frac{3m}{2m+1}-2\ell} \right)=o\left(n^{1+\ell\left(\frac{3m+1}{2m+1}-2\right)}\right).\] In the third, we note that $(3m+1)/(2m+1)\leq3/2$. The last one holds since $\ell\geq3$. $\checkmark$}
    \[ \Theta\left(\binom{|X|}{6} \left(\frac{t}{n^2}\right)^{6\ell}\right) = \Theta\left(\left(\frac{n^{3m+1}}{r^{2m+1}}\frac{t^\ell}{n^{2\ell}}\right)^6 \right) = o\left(\left(n^{1+\ell\left(\frac{3m+1}{2m+1}-2\right)} \right)^6\right)=o\left(n^{6-3\ell}\right) = o(n^{-\ell}).\]
    Hence, by a union bound, a.a.s.\ for every $\ell$-set $U\in\inbinom{V}{\ell}$ there are at most five vertices $x\in X$ such that $U\subseteq N_{F_1}(x)$.
    
    In a similar fashion, we claim that a.a.s. 
    \begin{enumerate}[label=$(\mathrm{CN}2)$]
        \item\label{item:common_neighs_2KT} every pair of vertices in~$\inbinom{V}{2}$ is contained in the $F_1$-neighbourhood of at most $4m+1$\COMMENT{Certainly not optimal. $\checkmark$} vertices $x\in X$.
    \end{enumerate}
    To show this, fix an arbitrary pair of vertices $U\in\inbinom{V}{2}$.     By the bounds on $t$ and $r$, with calculations analogous to those for proving \ref{item:common_neighs_ellKT}, we conclude that the probability that $U$ is contained in the $F_1$-neighbourhood of at least $4m+2$ vertices is at most\COMMENT{We have 
    \[\Theta\left(\binom{|X|}{4m+2}(p_1)^{8m+4}\right)
    =\Theta\left(\left(\frac{n^{3m+1}}{r^{2m+1}}\frac{t^2}{n^4}\right)^{4m+2}\right)=o\left(\left(n^{3m-3+\frac{6m+2}{2m+1}-3m}\right)^{4m+2}\right)=o\left(\left(n^{-\frac{1}{2m+1}}\right)^{4m+2}\right)=o(n^{-2}),\]
    where the first inequality simply uses the value of $|X|$, in the second we use the bounds on $t$ and $r$, and the rest are simple calculations. $\checkmark$} 
    \[\Theta\left(\binom{|X|}{4m+2}\left(\frac{t}{n^2}\right)^{8m+4}\right)=\Theta\left(\left(\frac{n^{3m+1}}{r^{2m+1}}\frac{t^2}{n^4}\right)^{4m+2}\right)=o\left(\left(n^{\frac{6m+2}{2m+1}-3}\right)^{4m+2}\right)=o(n^{-2}).\]
    The conclusion follows by a union bound over all possible pairs of vertices.
    
    Now condition on the event that \ref{item:common_neighs_ellKT} holds for all $2\leq\ell\leq m$ and that the bounds in \eqref{equa:K4sparse_e1H} hold (both of which occur a.a.s.).
    Let $Z$ denote the number of copies of\/ $T$ in $F_2$ whose vertex set is contained in $N_{F_1}(x)$ for some $x\in X$.
    Then, again by our choice of $r$,
    \[ \mathbb{E}[Z] = \Theta\left(|X|\left(\frac{t}{n}\right)^{m+1}\left(\frac{t}{n^2}\right)^m \right) = \Theta\left(\frac{n^{3m+1}}{r^{2m+1}}\frac{t^{2m+1}}{n^{3m+1}}\right) = \omega(1).\]
    Similarly as in the proof of \cref{thm:WheelStrategy}, we also have that\COMMENT{
    We next bound the variance by using the covariances of pairs of trees.
    Grouping all pairs of trees into those that intersect in exactly $j$ edges for each $j$ and noting that any pair of trees that overlap in exactly $j$ edges must have at least $j+1$ vertices in common, and using \ref{item:common_neighs_ellKT} to note that, upon fixing any set of $j+1$ vertices, the trees which we consider can only be contained in the neighbourhood of constantly many vertices in $X$, we conclude that
    \begin{align*}
        \mathrm{Var}(Z)&\leq \mathbb{E}[Z] + \sum_{j=1}^O\left(|X|\binom{2t/n}{m+1}\binom{2t/n}{m-j}\left(\frac{t}{n^2}\right)^{2m-j}\right)=\mathbb{E}[Z] + \sum_{j=1}^mO\left(\frac{n^{3m+1}}{r^{2m+1}}\frac{t^{2m+1}t^{2(m-j)}}{n^{6m-3j+1}}\right)\\
        &=\mathbb{E}[Z] \left(1+\sum_{j=1}^mO\left(\frac{t^{2(m-j)}}{n^{3(m-j)}}\right)\right).
    \end{align*}
    Using the upper bound on $t$ and the fact that $(3m+1)/(2m+1)<3/2$, we have that
    \[\frac{t^{2(m-j)}}{n^{3(m-j)}}=O\left(n^{2(m-j)\frac{3m+1}{2m+1}-3(m-j)}\right)=o(1).\]
    Hence, we have what we want.}
    \[\mathrm{Var}(Z)\leq(1+o(1))\mathbb{E}[Z]\]
    and, by Chebyshev's inequality, a.a.s.\ during the second stage Builder will claim at least one copy of~$T$ which completes a copy of $K_{1,T}$.
    
    \textbf{Case 2.}
    Suppose now that $t=\omega(n^{\frac{3m+1}{2m+1}})$.
    Note that in this range we have $b = \omega((n^2/t)^{m/(m+1)})$.
    Fix any such~$b$ and let $\tilde{n}=\tilde{n}(n)$ and $t_1=t_1(n)$ be such that $\tilde{n}=o(b)$, $\tilde{n}=o(t/n)$, $\tilde{n}=\omega((n^2/t)^{m/(m+1)})$, $\tilde{n}n\leq t_1=o(t)$, and $t_1=\omega(n^2/\tilde{n}^2)$.\COMMENT{Again, one needs to verify that such a choice is possible. Note that $(n^2/t)^{m/(m+1)} = o(t/n) \iff t = \omega(n^{(3m+1)/(2m+1)})$, which is our regime, so we can choose such $\tilde{n}$.
    For $t_1$, we have that $\tilde{n}n=o(t)$, so it suffices to verify that $t=\omega(n^2/\tilde{n}^2)\iff\tilde{n}=\omega((n^2/t)^{1/2})$, which holds as $n^2>t$ and $\tilde{n}=\omega((n^2/t)^{m/(m+1)})$. $\checkmark$}
    Now consider the strategy outlined in \cref{strat:KTdense}.
    For sufficiently large~$n$, this is a $(t,b)$-strategy by construction, so it only remains to prove that it is successful.

    \begin{algorithm}
    \caption{A $(t,b)$-strategy for $K_{1,T}$ for $t=\omega(n^{\frac{3m+1}{2m+1}})$.}\label{strat:KTdense}
    \begin{algorithmic}[1]
        \State{Fix a vertex $x\in[n]$.
        For time $t_1$ and while the built graph has at most $\tilde{n}$ edges, buy any presented edge with one endpoint being $x$.}
        \State{For time $t/2$, simulate an optimal $(t\tilde{n}^2/4n^2,b/2)$-strategy for constructing a copy of $T$ on $N_{B_{t_1}}(x)$.}
    \end{algorithmic}
    \end{algorithm}
    
    We can use the same reasoning as in the proof of Case~2 of \cref{thm:WheelStrategy} with $t'\coloneqq t\tilde{n}^2/4n^2$, using \cref{lem:TreeStrat} instead of \cref{lem:CycleStrat}.
    Note that by our choice of~$\tilde{n}$ we have that $b = \omega(\max\{(\tilde{n}/t')^{m-1},1\})$.\COMMENT{Indeed, it is obvious that $b=\omega(1)$ since $t=O(n^2)$, so it remains to show that $b = \omega((\tilde{n}/t')^{m-1})$ for $m\geq2$, as it holds trivially when $m=1$.
    To see this, it suffices to prove the second inequality below: 
    \[b=\omega\left(\left(\frac{n^2}{t}\right)^{\frac{m}{m+1}}\right)=\omega\left(\left(\frac{\tilde{n}}{t'}\right)^{m-1}\right)=\omega\left(\left(\frac{n^2}{t\tilde{n}}\right)^{m-1}\right).\]
    For the inequality we want, it suffices to verify that
    \[\tilde{n}^{m-1}=\Omega\left(\frac{n^{2m-2-2\frac{m}{m+1}}}{t^{m-1-\frac{m}{m+1}}}\right)\iff\tilde{n}=\Omega\left(\left(\frac{n^2}{t}\right)^{1-\frac{m}{(m-1)(m+1)}}\right).\]
    Since $\tilde{n}=\omega((n^2/t)^{m/(m+1)})$, it suffices to verify that
    \[\frac{m}{m+1}\geq1-\frac{m}{(m-1)(m+1)}\iff\frac{(m-1)m}{(m-1)(m+1)}\geq\frac{(m-1)(m+1)-m}{(m-1)(m+1)}\iff0\geq-1,\]
    which holds, as we wanted to see.}
    Thus, by \cref{lem:TreeStrat}, we conclude that a.a.s.\ we construct a copy of $T$ inside $N_{B_{t_1}}(x)$, which results in a copy of $K_{1,T}$.
\end{proof}

\end{document}